\documentclass[12pt]{article}
\usepackage{geometry}                
\usepackage{graphicx}
\usepackage{amssymb}
\usepackage{amsmath}
\usepackage{mathtools}
\usepackage{enumerate}
\usepackage{mathrsfs}
\usepackage{xcolor}
\usepackage{amsthm}
\usepackage{hyperref}
\usepackage[style=alphabetic]{biblatex}
\AtEveryBibitem{
    \clearfield{urlyear}
    \clearfield{urlmonth}
    \clearfield{isbn}
    \clearfield{issn}
    \clearfield{pagetotal}
}

\newcommand*{\FF}{\mathbb{F}}
\newcommand*{\NN}{\mathbb{N}}

\newcommand*{\RR}{\mathbb{R}}
\newcommand*{\CC}{\mathbb{C}}

\newcommand*{\PP}{\mathbb{P}}
\DeclareMathOperator*{\EE}{\mathbb{E}}

\newcommand*{\calE}{\mathcal{E}}
\newcommand*{\calF}{\mathcal{F}}

\newcommand*{\calO}{\mathcal{O}}
\newcommand*{\calU}{\mathcal{U}}

\newcommand*{\A}{\mathtt{A}}
\newcommand*{\ta}{\mathtt{a}}
\newcommand*{\B}{\mathtt{B}}
\newcommand*{\tb}{\mathtt{b}}

\newcommand*{\1}{\mathbf{1}}

\newcommand*{\mb}[1]{\mathbf{#1}}

\newcommand*{\st}{\,:\,}

\newcommand*{\ball}[3][\relax]{\mathrm{B}^{#1}(#2, #3)}

\newtheorem{lemma}{Lemma}[section]
\newtheorem{cor}[lemma]{Corollary}
\newtheorem{prop}[lemma]{Proposition}
\newtheorem{theorem}[lemma]{Theorem}

\theoremstyle{definition}

\DeclareMathOperator{\Prob}{Prob}
\DeclareMathOperator{\Var}{Var}

\newcommand*{\unif}{\mathrm{up}}

\DeclareMathOperator{\Hom}{Hom}

\DeclareMathOperator{\Sym}{Sym}

\DeclareMathOperator{\vspan}{span}
\DeclareMathOperator{\Lip}{Lip}

\DeclareMathOperator{\diam}{diam}
\DeclareMathOperator{\shent}{H}
\DeclareMathOperator{\h}{h}

\DeclareMathOperator{\f}{f}

\DeclareMathOperator{\press}{p}

\DeclareMathOperator{\acoth}{arccoth}

\DeclarePairedDelimiter{\abs}{\lvert}{\rvert}

\DeclarePairedDelimiter{\norm}{\|}{\|}
\newcommand{\nnorm}[1]{{\left\vert\kern-0.25ex\left\vert\kern-0.25ex\left\vert #1 
    \right\vert\kern-0.25ex\right\vert\kern-0.25ex\right\vert}}
\DeclarePairedDelimiterX{\inprod}[2]{\langle}{\rangle}{#1,\ #2}

\newcommand*{\sofic}{\mathrm{sofic}}
\newcommand*{\ann}{\mathrm{ann}}

\title{Typical sofic entropy and local limits for free group shift systems}
\author{Christopher Shriver}

\begin{document}
\maketitle

\begin{abstract}
	We show that for any invariant measure $\mu$ on a free group shift system, there are two numbers $h^\flat \leq h^\sharp \in \{-\infty\} \cup [0, \log \abs{\A}]$ which in some sense are the typical upper and lower sofic entropy values. We also give a condition under which $h^\flat = h^\sharp = \f(\mu)$. This can be used to compute typical local limits of finitary Gibbs states over sequences of random regular graphs. As examples, we work out typical local limits of the Ising and Potts models.
	
	We also show that, for Markov chains, the Kesten--Stigum second-eigenvalue reconstruction criterion actually implies there are no good models over a typical sofic approximation (i.e. $h^\sharp = -\infty$). In particular, we have an exact value for the typical entropy $h^\flat = h^\sharp$ of the free-boundary Ising state: it is equal to the annealed entropy $\f$ for interaction strengths up to the reconstruction threshold, after which it drops abruptly to $-\infty$.
\end{abstract}

\tableofcontents

\section{Introduction}

Throughout, let $\A$ be a finite set (an ``alphabet''), $\Gamma$ be the free group of rank $r$, and $\mu \in \Prob^\Gamma(\A^\Gamma)$ be a shift-invariant measure. The shift action is defined precisely in Section~\ref{sec:definitions}. We will usually identify $\Gamma$ with its left Cayley graph, which is a $2r$-regular tree.

In this paper, we consider the question of whether a typical random $2r$-regular graph, having geometry locally like $\Gamma$, admits vertex labelings whose average local statistics are close to $\mu$. We also investigate how many such labelings a typical graph has and, as an application, work out some typical local limits of Gibbs states on random regular graphs.

Before stating the main results, we give a brief introduction to sofic entropy. See Section~\ref{sec:definitions} below for more precise definitions.

A sofic approximation to $\Gamma$ is a sequence $\Sigma$ of $2r$-regular graphs which converge locally to the Cayley graph of $\Gamma$. These graphs may be random or deterministic. In the deterministic case, $\overline{\h}_\Sigma(\mu)$ is the upper exponential growth rate of the number of vertex labelings of these graphs whose average local statistics are consistent with $\mu$. We call these labelings \emph{microstates} or \emph{good models} for $\mu$. If the graphs are random, then $\overline{\h}_\Sigma$ refers to the upper exponential growth rate of the \emph{expected} number of such labelings. Similarly, $\underline{\h}_\Sigma$ denotes the lower exponential growth rate. Sofic entropy was introduced by Lewis Bowen as an isomorphism invariant which distinguishes between Bernoulli shifts (iid processes) over nonamenable groups \cite{bowen2010a}.

An important special case is when the random graphs are drawn from the uniform permutation model $\PP^\unif$: given a finite set $V$, pick $r$ permutations in $\Sym(V)$ uniformly and independently. Then produce a graph with vertex set $V$ by connecting each $v \in V$ to its images under the $r$ permutations. This tuple of $r$ permutations can also be thought of as an element of $\Hom(\Gamma, \Sym(V))$, where the $i$th generator of $\Gamma$ acts on $V$ by the $i$th permutation. If $\Sigma$ is this sequence of random graphs, $\overline{\h}_\Sigma = \underline{\h}_\Sigma$ is called the $\f$ invariant \cite{bowen2010c, bowen2010}, which we will denote by $\f(\cdot)$.

We will state some results for a more general kind of random sofic approximation which we call ``exponentially concentrated random sofic approximations'' and call the associated upper and lower ``annealed'' entropies $\overline{\h}^\ann(\cdot)$ and $\underline{\h}^\ann(\cdot)$. See Section~\ref{sec:ecrsa} for definitions. In the present paper we will only establish that the uniform permutation model has this property, but the stochastic block models introduced in \cite{shriver2022b} are other likely candidates.

\subsection{Typical sofic entropy values}

We say that $\mu$ satisfies the \emph{second-moment criterion} for a random sofic approximation if for every joining $\lambda$ of $\mu$ with itself we have $\overline{\h}^\ann(\lambda) \leq 2 \underline{\h}^\ann(\mu)$; recall that a joining is a shift-invariant coupling.

In the case of the uniform permutation model, an equivalent formulation is that the product $\mu \times \mu$ has maximal $\f$ invariant among all self-joinings of $\mu$. Note that we do not require the product to be the \emph{only} joining with maximal $\f$. For Gibbs measures of nearest-neighbor interactions with no hard constraints, the second-moment criterion is implied by non-reconstruction \cite[Corollary 16]{shriver2021}. \\


In the uniform case, Proposition 4.1 of \cite{shriver2023a} states that if $\mu$ satisfies the second-moment criterion then $\underline{\h}_\Sigma(\mu) = \overline{\h}_\Sigma(\mu) = \f(\mu)$ for a ``typical'' sofic approximation $\Sigma$ to $\Gamma$. In general, when we say a property holds for a typical sofic approximation, we mean that there is a sequence sets of homomorphisms $\{ S_n \subset \Hom(\Gamma, \Sym(n))\}_{n=1}^\infty$ such that $\PP^\unif(S_n) \to 1$ and if $\sigma_n\in S_n$ for each $n$ then $\Sigma = \{\sigma_n\}_{n=1}^\infty$ is a sofic approximation with that property.

This terminology somewhat obscures the order of quantifiers. In \cite{shriver2023a}, the possibility of the sequence $\{S_n\}_{n=1}^\infty$ depending on $\mu$ is not ruled out.

The first result of the present paper is a new version which is stronger in a few ways: it shows that $\{S_n\}_{n=1}^\infty$ can be chosen independent of $\mu$. Second, it shows there are typical upper and lower sofic entropy values for $\mu$ even if $\mu$ does not satisfy the second-moment criterion.

We also allow any random sofic approximation which is exponentially concentrated; see Section~\ref{sec:ecrsa} for definitions. Note, however, that we only verify this property for the uniform permutation model $\PP^\unif$.

\begin{theorem}
\label{thm:main}
	Let $\PP$ denote an exponentially concentrated random sofic approximation to $\Gamma$ and let $\A$ be a finite alphabet. Then there exist upper semicontinuous functions $h^\sharp,h^\flat \colon \Prob^\Gamma(\A^\Gamma) \to \{-\infty\} \cup [0, \log\abs{\A}]$ with the following property:
	
	There exists a sequence $\{S_n \subset \Hom(\Gamma, \Sym(n)\}_{n=1}^\infty$ such that $\PP(S_n) \to 1$ and if $\sigma_n \in S_n$ for each $n$ then $\Sigma = \{\sigma_n\}_{n=1}^\infty$ is a sofic approximation, and for all $\mu \in \Prob^\Gamma(\A^\Gamma)$ we have
		\[ \overline\h_\Sigma(\mu) = h^\sharp(\mu) \]
	and
		\[ \underline\h_\Sigma(\mu) = h^\flat(\mu) . \]
\end{theorem}

We prove Theorem~\ref{thm:main} in Section~\ref{sec:typicalproof}.

For the uniform permutation model, for every $\mu$ we have $h^\flat(\mu) \leq h^\sharp(\mu) \leq \f(\mu)$ (Lemma~\ref{lem:fupperbound}), and if $\mu$ satisfies the second-moment criterion, then the three quantities are equal (Cor.~\ref{cor:2ndmomentf}).

For more general sequences, since the upper and lower annealed entropies are not always equal, in general we have $h^\flat(\mu) \leq \underline{\h}^\ann(\mu)$ and $h^\sharp(\mu) \leq \overline{\h}^\ann(\mu)$ (Lemma~\ref{lem:fupperbound}), and the second-moment criterion implies $h^\flat(\mu) = \underline{\h}^\ann(\mu)$ and $h^\sharp(\mu) =\overline{\h}^\ann(\mu)$ (Prop.~\ref{prop:secondmomentgeneral}).

\begin{cor}
	For every $\PP$ and $\A$, the typical entropies $h^\sharp$ and $h^\flat$ are isomorphism-invariant.
\end{cor}
\begin{proof}
	This follows from Theorem~\ref{thm:main} and isomorphism-invariance of $\overline{\h}_\Sigma,\underline{\h}_\Sigma$ for every sofic approximation $\Sigma$.
\end{proof}

\subsection{Typical local limits}
For precise definitions of terminology used here, see Section~\ref{sec:locallimits}.

We use Theorem~\ref{thm:main} to give a method for establishing typical local limits of nearest-neighbor interactions. Given such an interaction and a sofic approximation $\Sigma = \{\sigma_n\}_{n=1}^\infty$, we want to understand the local statistics of the Gibbs states $\xi_n$ on the finite graphs $\sigma_n$. If $n$ is large, then the graph of $\sigma_n$ is locally tree-like at most vertices, so the marginal of the $\xi_n$ on the radius-$r$ neighborhood of most vertices can be compared to some fixed measure $\mu \in \Prob(\A^\Gamma)$. 

If, for a typical sofic approximation, the average of these marginals converges to $\mu$ in the weak topology as $n \to \infty$, we say that $\mu$ is the typical local-on-average limit.

If all but a vanishing fraction of the marginals are close to $\mu$ as $n \to \infty$, we say that $\mu$ is the typical local limit. We will establish this stronger mode of convergence.

The idea is the following: Theorem \ref{thm:main} implies that if a measure $\mu$ is $\f$-equilibrium and satisfies the second-moment criterion, then it's also equilibrium over typical $\Sigma$. This is because if $\nu\in \Prob^\Gamma(\A^\Gamma)$ is any other state then
	\[ \press_\Sigma(\nu) = \overline\h_\Sigma(\nu) - u(\nu) = h^\sharp(\nu) - u(\nu) \leq \f(\nu) - u(\nu) = \press_{\f}(\nu) \leq \press_{\f}(\mu) = \cdots = \press_\Sigma(\mu) . \]
	If $\mu$ is the unique $\f$-equilibrium state then the same argument shows that $\mu$ is uniquely $\Sigma$-equilibrium over a typical $\Sigma$. By \cite[Theorem~6.5]{shriver2023a}, this means it must be the local-on-average limit of Gibbs states over a typical sofic approximation.

In Section~\ref{sec:locallimits} we work out a few typical local limits over the uniform permutation model $\PP^\unif$ using a refined version of this method. After proving a general result (Proposition~\ref{prop:locallimit}), as a first example (Proposition~\ref{prop:isinglimit}) we show how this makes the ferromagnetic Ising case (previously established in \cite{montanari2012}) easy. The antiferromagnetic case, which does not seem to be in the literature, is also easily dealt with (Proposition~\ref{prop:afisinglimit}).

Then we use calculations from \cite{galanis2016a} to work out the $\PP^\unif$-typical local limit of the Potts model at all temperatures but one -- at the critical temperature where the ordered and disordered states have the same pressure, we are unable to say exactly what the limit is (Proposition~\ref{prop:pottslimit}).

Note that this method relies on there being only a few equilibrium states among the uncountably many Gibbs states, which is a special property of nonamenable systems: for nice enough interactions over amenable groups, all Gibbs states are equilibrium \cite{tempelman1984}.


\subsection{Kesten--Stigum bound for typicality}
\label{sec:KSintro}
In this section we consider only the uniform permutation model.

As remarked above, if $\mu$ satisfies the second-moment criterion for the uniform permutation model then $h^\flat(\mu) = h^\sharp(\mu) = \f(\mu)$ (Cor.~\ref{cor:2ndmomentf}). If $\mu$ is a Markov chain then this gives a simple formula for the typical entropy (see Section~\ref{sec:finvariant}). To complement this, we give a criterion which implies the typical entropy of a $\Gamma$-indexed Markov chain is $-\infty$.

We call a measure $\mu$ \emph{weakly typical} if $h^\sharp(\mu) > -\infty$. This is analogous to the notion of typicality introduced in \cite{backhausz2018}; see also \cite{backhausz2022}.

Given a Markov chain $\mu$ indexed by $\Gamma=\FF_r$, let $\theta$ denote the second-largest (in absolute value) eigenvalue of its transition matrix. The Kesten--Stigum bound for reconstruction states that if $\abs{\theta}^2 (2r-1) > 1$ then reconstruction is possible. (In fact, this is the exact threshold for ``census'' reconstruction \cite{mossel2003}.) 
We prove the following stronger version of this statement:

\begin{theorem}
\label{thm:KSbound}
	Suppose $\mu$ is a Markov chain indexed by $\Gamma$ and $\theta$ is the second-largest eigenvalue of its transition matrix. If $\abs{\theta}^2(2r-1) > 1$ then $\mu$ is not weakly typical.
\end{theorem}

We describe this as stronger because, for nearest-neighbor interactions with no hard constraints (which in this case means the Markov transition matrix has all positive entries), nonreconstruction implies the second moment criterion \cite[Corollary 16]{shriver2021} which implies typicality (by Theorem~\ref{thm:main} or \cite[Proposition~4.1]{shriver2023a}).

We actually show something which may be even stronger: that if $\abs{\theta}^2(2r-1) > 1$ then the Koopman representation on $L_0^2(\mu)$ is not weakly contained in the left-regular representation of $\Gamma$. Here $L_0^2(\mu)$ is the Hilbert space of square-integrable mean-zero functions on $(\A^\Gamma, \mu)$. This weak containment is implied by typicality (Thm.~\ref{thm:repopbound}) but we do not know if the converse is true.

The proof is related to proofs of some correlation bounds for factors of iid processes \cite{backhausz2015,lyons2017}. \\

As an example, we apply this to the Ising model: let $\A = \{-1,+1\}$, and for each $\theta \in (-1/2, 1/2)$ let $\mu_\theta \in \Prob^\Gamma(\A^\Gamma)$ denote the stationary Markov chain with uniform single-site marginal and transition matrix
	\[ \begin{pmatrix}
			\frac{1+\theta}{2} & \frac{1-\theta}{2} \\
			\frac{1-\theta}{2} & \frac{1+\theta}{2} 
		\end{pmatrix} . \]
This is a free boundary state with inverse temperature determined by $\theta$. If $\theta<0$ then the model is considered antiferromagnetic.

This parametrization is convenient because $\theta$ is the second-largest eigenvalue of the transition matrix.

\begin{cor}
\label{cor:isingtypical}
	If $\abs{\theta} > \frac{1}{\sqrt{2r-1}}$ then $\mu_\theta$ is not weakly typical.
\end{cor}
	Conversely, if $\abs{\theta} \leq \frac{1}{\sqrt{2r-1}}$ then $\mu_\theta$ is tail-trivial \cite{bleher1995} so satisfies the second-moment criterion \cite[Corollary 16]{shriver2021}. Hence the Kesten--Stigum criterion gives an exact threshold for typicality of the free-boundary Ising state. Combining the previous corollary with Corollary~\ref{cor:2ndmomentf} gives
		\[ h^\flat(\mu_\theta) = h^\sharp(\mu_\theta)
			= \left\{ \begin{array}{ll}
				\log 2 + r (\shent(\tfrac{1-\theta}{2}) - \log 2), & \theta^2(2r-1) \leq 1 \\
				-\infty, & \text{else}
				\end{array} \right. \]
	where $\shent(t) = -t \log t - (1-t)\log (1-t)$. See Figure~\ref{fig:isingentropy} for a plot.
	
\begin{figure}
\begin{center}
	\includegraphics[width=0.7\textwidth]{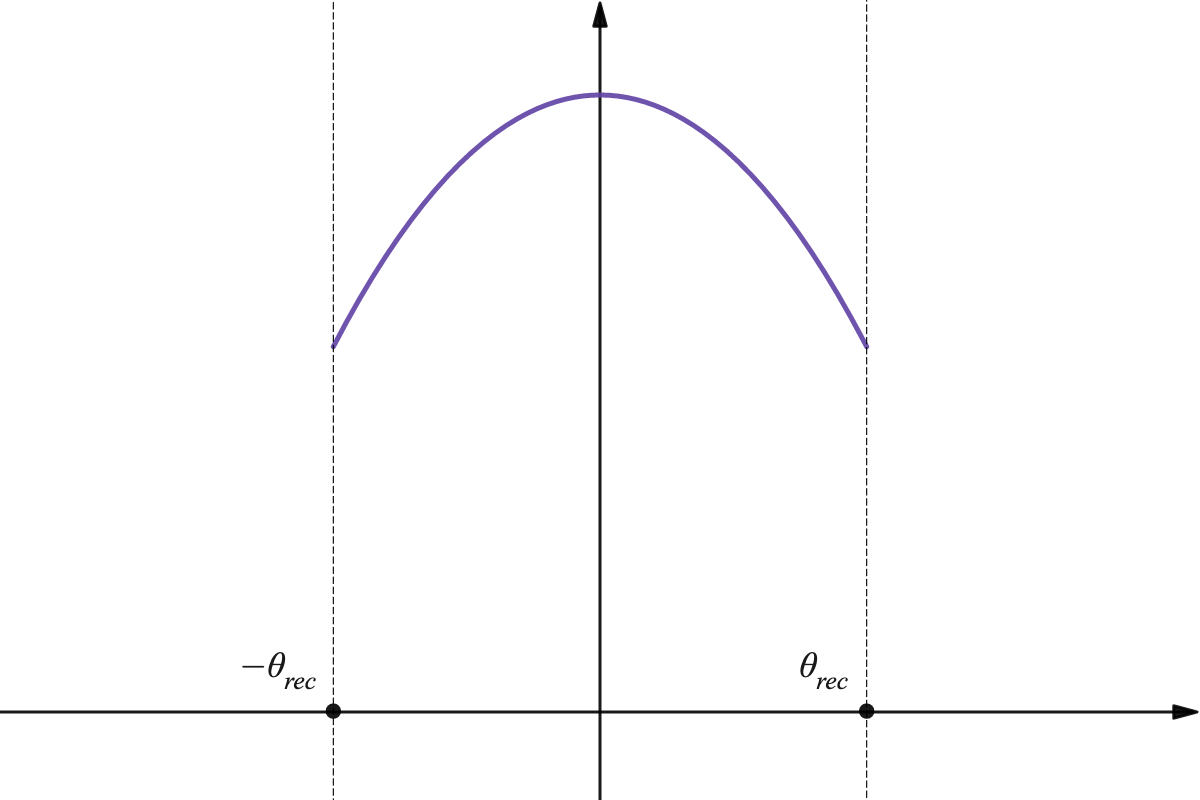}
	\caption{Typical entropy of free boundary Ising state with parameter $\theta \in (-1/2,\, 1/2)$. Outside the $\theta$ values marked with dashed lines, the value drops to $-\infty$.}
	\label{fig:isingentropy}
\end{center}
\end{figure}

	Any good model for $\mu_\theta$ on a finite graph $(V,E)$ determines a bisection of the vertex set into two approximately equal-sized pieces with about $\frac{1-\theta}{2}$ fraction of edges going from one piece to the other. If $\theta$ is close to 1, then, it is intuitively clear that there should typically be no good models. On the other hand, every such bisection has approximately the edge statistics of $\mu_\theta$, but may have different statistics over general neighborhoods, so is a good model for some measure with Markov approximation $\mu_\theta$ but not necessarily $\mu_\theta$ itself.

	One of the main results of \cite{dembo2017} determines the asymptotic minimal cut density of a random regular graph. It follows from this and Corollary~\ref{cor:isingtypical} that there is some range of $\theta$ for which there are typically good models for \emph{some} measure whose Markov approximation is $\mu_\theta$, but no good models for $\mu_\theta$ itself (at least for large $r$). Hence, given an edge marginal, the Markov chain is not necessarily the ``most typical'' measure with that edge marginal. \\

However, as is the case with reconstruction, the resulting bound for typicality of the hardcore model is \emph{not} sharp: Let $\A = \{1,0\}$. Given $\alpha \in [0,1/2]$, let $\mu_\alpha \in \Prob^\Gamma(\A^\Gamma)$ be the Markov chain with single-site marginal $(\alpha,\ 1-\alpha)$ and transition matrix
	\[ \begin{pmatrix}
			0 & \frac{\alpha}{1-\alpha} \\
			1 & \frac{1-2\alpha}{1-\alpha} 
		\end{pmatrix} . \]
This has second-largest eigenvalue $-\frac{\alpha}{1-\alpha}$. If we identify elements of $\{1,0\}^\Gamma$ with subsets of $\Gamma$, then $\mu_\alpha$ is supported on independent sets (subsets containing no two adjacent vertices). Any good model for $\mu_\alpha$ is a set of vertices of a finite graph with density $\alpha$ and a vanishing fraction of adjacent vertices.
\begin{cor}
	If $\frac{\alpha}{1-\alpha} > \frac{1}{\sqrt{2r-1}}$ (equivalently, if $\alpha > \frac{1}{1+\sqrt{2r-1}}$) then the hardcore free-boundary state $\mu_\alpha$ is not weakly typical.
\end{cor}

This bound on the density $\alpha$ is not optimal, at least for large $r$, since it's larger than the maximal density of an independent set which is $2\frac{\log r}{r}+o_{r \to \infty}(1)$ \cite{frieze1992}. 


\subsection{Acknowledgements}
This material is based on work completed at the University of Texas at Austin while supported by NSF Grant DMS 1937215.

Thanks to Tim Austin and Lewis Bowen for helpful conversations and for comments on earlier drafts.

\section{Definitions and basic lemmas}
\label{sec:definitions}

Throughout, let $\A$ be a finite alphabet, $\Gamma$ be the free group of rank $r$ with generating set $\{s_1, \ldots, s_r\}$ and identity $e$.

The shift action of $\Gamma$ is defined as follows: if $\mb{x} \in \A^\Gamma$ and $g,h \in \Gamma$ then the label of $g \cdot \mb{x}$ at $h$ is given by
	\[ (g \cdot \mb{x})(h) = \mb{x}(hg) . \]
This is a left action. The action of $g \in \Gamma$ on $\Prob^\Gamma(\A^\Gamma)$ is then given by pushforwards:
	\[ g \cdot \mu = (\mb{x} \mapsto g\cdot \mb{x})_* \mu . \]

We will identify $\Gamma$ with its left Cayley graph, which is a $2r$-regular tree with edges of the form $(g,s_i g)$ for $g \in \Gamma$ and $i \in [r] = \{1, 2, \ldots, r\}$. 

\subsection{Sofic entropy}
We follow the notation conventions of \cite{austin2016}, but for simplicity assume the sofic approximation maps are homomorphisms.

Given finite set $V$ and a homomorphism $\sigma \in \Hom(\Gamma, \Sym(V))$ and a ``microstate'' $\mb{x} \in \A^V$, we define the pullback labeling of $\mb{x}$ at $v \in V$ as the labeling $\Pi_v^\sigma \mb{x} \in \A^\Gamma$ defined by
	\[ \left( \Pi_v^\sigma \mb{x} \right)(h) = \mb{x}(\sigma^h v). \]
Note that with this definition we have $g \cdot (\Pi_v^\sigma \mb{x}) = \Pi_{\sigma^g v}^\sigma \mb{x}$, since $\sigma$ is a homomorphism.

The empirical distribution of $\mb{x}$ over $\sigma$ is then defined by
	\[ P_{\mb{x}}^\sigma = \frac{1}{\abs{V}} \sum_{v \in V} \delta_{\Pi_v^\sigma \mb{x}} \in \Prob^\Gamma(\A^\Gamma) . \]
Since $\sigma$ is a homomorphism, the empirical distribution is always shift-invariant.
	
Briefly, the sofic entropy of $\mu \in \Prob^\Gamma(\A^\Gamma)$ over a sequence of homomorphisms is the exponential growth rate of the number of microstates whose empirical distribution is close to $\mu$ in the weak topology. To make this precise, endow $\Prob^\Gamma(\A^\Gamma)$ with the transportation metric 
	\[ \overline{d}(\mu, \nu) = \sup \left\{ \abs*{\int f\, d\mu - \int f\, d\nu} \st f \in \Lip_1(\A^\Gamma) \right\} \]
where $\Lip_1(\A^\Gamma)$ is the set of real-valued, 1-Lipschitz functions on $\A^\Gamma$, which we give the metric
	\[ d(\mb{x}, \mb{y}) = \sum_{\gamma \in \Gamma} \1_{\{\mb{x}(\gamma) \ne \mb{y}(\gamma)\}} (5r^2)^{-d(\gamma, e)}. \]
For $\sigma \in \Hom(\Gamma, \Sym(V))$ we then define the set of $\varepsilon$-good microstates for $\mu$ by
	\[ \Omega(\sigma, \mu, \varepsilon) = \{ \mb{x} \in \A^V \st \overline{d}( P_{\mb{x}}^\sigma,\ \mu) < \varepsilon \}. \]
Here we could alternatively use open neighborhoods of $\mu$ instead of fixing a metric, but using a metric will be convenient for defining $h^\flat,h^\sharp$. The choice of $5r^2$ in particular will be convenient in the proof of Lemma~\ref{lem:rhoconcentration}.
	
We will consider both upper and lower exponential growth rates: if $\Sigma = \{ \sigma_n \in \Hom(\Gamma, \Sym(V_n))\}_{n=1}^\infty$ the upper sofic entropy over $\Sigma$ is defined by
	\[ \overline{\h}_\Sigma(\mu) = \inf_{\varepsilon>0} \limsup_{n \to \infty} \frac{1}{\abs{V_n}} \log \abs{\Omega(\sigma_n, \mu, \varepsilon)} \]
and the lower sofic entropy by
	\[ \underline{\h}_\Sigma(\mu) = \inf_{\varepsilon>0} \liminf_{n \to \infty} \frac{1}{\abs{V_n}} \log \abs{\Omega(\sigma_n, \mu, \varepsilon)} . \]

For a finite set $F \subset \Gamma$ and $\delta>0$, a homomorphism $\sigma$ is called $(F,\delta)$-sofic if 
	\[ \frac{1}{\abs{V}} \abs*{\{v \in V \st \sigma(\gamma_1)\, v \ne \sigma(\gamma_1)\,v \ \text{for all distinct } \gamma_1,\gamma_2 \in F \}} > 1- \delta . \]
For example, if $F$ is the radius-$R$ ball centered at the identity, then this requires that for all but a fraction $\delta$ of the vertices $v$, the subgraph of $\sigma$ induced by the vertices within distance $R$ of $v$ is a tree. In other words, the graph of $\sigma$ locally looks like the Cayley graph of $\Gamma$, with quality parameters $(F,\delta)$.

The sequence $\Sigma$ is called a sofic approximation to $\Gamma$ if for any $(F, \delta)$, for all large enough $n$ $\sigma_n$ is $(F,\delta)$-sofic. If $\Sigma$ is a sofic approximation then these quantities are measure-conjugacy invariants \cite{bowen2010a}.

The quantities may depend on the choice of sofic approximation $\Sigma$, but it is not fully understood how. It is known that for some measures, the value can be $-\infty$ for some choices of $\Sigma$ and positive for others. But it is not known whether two choices of $\Sigma$ can achieve distinct positive values for the same measure.

\subsection{$\f$ invariant}
\label{sec:finvariant}
The $\f$ invariant is the exponential growth rate of the expected number of good microstates over a uniformly random homomorphism: if $\sigma_n$ is uniformly distributed over $\Hom(\Gamma, \Sym(n))$ then
	\[ \f(\mu) \coloneqq \inf_{\varepsilon > 0} \limsup_{n \to \infty} \frac{1}{n} \log \EE \abs{\Omega(\sigma_n, \mu, \varepsilon)} . \]
The method of \cite{bowen2010c} shows that the limsup can be replaced by a liminf without changing the value: in either case $\f$ is equivalently given by the information-theoretic formula in \cite{bowen2010}.

A result of \cite{bowen2010b} is that there is an especially simple formula for Markov chains: if $p_1 \in \Prob(\A)$ is the single-site marginal of $\mu$ and $p_2 \in \Prob(\A^2)$ is the marginal on any pair of adjacent elements of $\Gamma$, then
	\[ \f(\mu) = (1-2r) \shent(p_1) + r \shent(p_2) \]
where $\shent$ is the Shannon entropy.

The $\f$ invariant is also sometimes called the annealed entropy, but in the present paper we will use that term for sofic entropy over more general random sofic approximations. We use the term $\f$ invariant to make more clear when we are discussing the uniform case specifically.


\subsection{Exponentially concentrated random sofic approximations}
\label{sec:ecrsa}

A \emph{random sofic approximation} (by homomorphisms) is a sequence of random homomorphisms $\{\sigma_n \in \Hom(\Gamma, \Sym(V_n))\}_{n=1}^\infty$ such that for every $F,\delta$, the random homomorphism is $(F,\delta)$-sofic with superexponentially high probability:
	\[ \limsup_{n \to \infty} \frac{1}{\abs{V_n}} \log \PP\{ \sigma_n \text{ is not $(F,\delta)$-sofic}\} = - \infty . \]
This implies the existence of a sequence of sets of homomorphisms $\{ S_n^\sofic \}_{n=1}^\infty$ such that $\PP(S_n^\sofic) \to 1$ and if $\sigma_n \in S_n^\sofic$ for all $n$ then $\{\sigma_n\}$ is a sofic approximation. For the proof of Theorem~\ref{thm:main}, we will only need to use this sequence, not the superexponential decay.

We say that two homomorphisms $\sigma_1, \sigma_2 \in \Hom(\Gamma, \Sym(V))$ differ by a ``switching'' if for some $i,j \in V$ and $k \in [r]$
	\[ \sigma_1^k(i) = \sigma_2^k(j) \quad\text{and}\quad \sigma_1^k(j) = \sigma_2^k(i) , \]
and $\sigma_1, \sigma_2$ are otherwise identical. If $\sigma_2$ can be obtained from $\sigma_1$ using a minimal number $m$ of switchings, then we define $d(\sigma_1, \sigma_2) = m/\abs{V}$.

We call a sequence of random homomorphisms \emph{exponentially concentrated} if there is some constant $C>0$ such that for every $n \in \NN$, 1-Lipschitz $f \colon \Hom(\Gamma, \Sym(V_n)) \to \RR$, and $t > 0$ we have
	\[ \PP\{ \abs{f(\sigma_n) - \EE [f(\sigma_n)]} > t \} \leq 2 e^{-C \abs{V_n} t^2} . \]
	
By \cite[Lemma 22]{shriver2021}, this holds for the uniform permutation model with $C = \frac{1}{2r}$. The proof uses Azuma's inequality, mimicking a similar concentration inequality for the configuration model which also uses the notion of ``switching'' \cite{wormald1999}.

The exact dependence of the upper bound on $t$ is not important, but we do not try to formulate the most general version here.


\subsection{Upper and lower typical entropy functions $h^\sharp$, $h^\flat$}
In this section, $\PP$ and $\EE$ refer to an arbitrary (but fixed throughout) exponentially concentrated random sofic approximation by homomorphisms.

For $h \in [0, \log|\A|]$ and $\sigma \in \Hom(\Gamma, \Sym(n))$, let
	\[ \rho_\mu (\sigma, h) = \inf \left\{ \varepsilon>0 \st \frac{1}{n} \log \abs*{\Omega(\sigma, \mu, \varepsilon)} \geq h \right\} . \]
Note that $\rho_\mu(\sigma, h) \leq \diam(\Prob^\Gamma(\A^\Gamma))$. Given a random sofic approximation we define the limits
	\[ \overline\rho_\mu (h) = \limsup_{n \to \infty} \EE \rho_\mu(\sigma_n, h) \]
and
	\[ \underline\rho_\mu (h) = \liminf_{n \to \infty} \EE \rho_\mu(\sigma_n, h) \]
where the expectations are over the random $\sigma_n \in \Hom(\Gamma, \Sym(V_n))$.

For each $\sigma$ the function $\rho_\mu(\sigma, h)$ is nondecreasing in $h$, so the expectations and limit functions are nondecreasing as well. Define
	\[ h^\flat = \sup \{ h \in \RR \st \overline\rho_\mu(h) = 0 \} \ \leq \ h^\sharp = \sup \{ h \in \RR \st \underline\rho_\mu(h) = 0 \} \ \leq \ \log \abs{\A} , \]
using the convention that the supremum of the empty set is $-\infty$.
	
For the permutation model, it may be that $h^\flat \equiv h^\sharp$, but this seems to be a difficult question; cf.~\cite[Conjecture~1.4]{backhausz2022}.

Note that since $\underline{\h}_\Sigma$ and $\overline{\h}_\Sigma$ are independent of the choice of metric on $\A^\Gamma$, Theorem~\ref{thm:main} implies that $h^\flat,h^\sharp$ are independent of the metric as well.


\subsubsection{Properties of $\rho$}

The crucial property of $\rho_\mu(\cdot, h)$ is that it's exponentially concentrated around its mean. This follows from the exponential concentration property of the random sofic approximation, once we show that it's a Lipschitz function:

\begin{lemma}
\label{lem:rhoconcentration}
For every $\mu \in \Prob^\Gamma(\A^{\Gamma})$ and $h>0$, the function $\sigma \mapsto \rho_\mu(\sigma, h)$ is $2$-Lipschitz.
In particular, over an exponentially concentrated random sofic approximation we have
	\[ \PP \left\{ \abs*{ \rho_\mu (\sigma_n, h) - \EE [\rho_\mu(\sigma_n, h)]} > t \right\} \leq 2 \exp\left( \frac{-C\abs{V_n} t^2}{4} \right) .\]
\end{lemma}
The constant $C$ here is the same one from the definition of exponential concentration above, and does not depend on $\mu$ or $h$.

\begin{proof}
	We first show that if $\sigma_1, \sigma_2 \in \Hom(\Gamma, \Sym(V))$ differ by a switching then for any microstate $\mb{x} \in \A^V$, its empirical distributions $P_{\mb{x}}^{\sigma_1}, P_{\mb{x}}^{\sigma_2}$ are distance at most $2/\abs{V}$ apart. By definition of the transportation metric,
		\[ \overline{d}(P_{\mb{x}}^{\sigma_1}, P_{\mb{x}}^{\sigma_2}) = 
			\sup\left\{ \abs*{ \int f\, d P_{\mb{x}}^{\sigma_1} - \int f\, d P_{\mb{x}}^{\sigma_2}} \st f \in \Lip_1(\A^\Gamma) \right\} \]
	where $\Lip_1(\A^\Gamma)$ is the set of real-valued, 1-Lipschitz functions on $\A^\Gamma$. For any such function, by definition of the empirical distribution,
		\[ \abs*{ \int f\, d P_{\mb{x}}^{\sigma_1} - \int f\, d P_{\mb{x}}^{\sigma_2}}
    			\leq \frac{1}{\abs{V}} \sum_{v \in V} \abs*{f(\Pi_v^{\sigma_1}\mb{x}) - f(\Pi_v^{\sigma_2}\mb{x})}
			\leq \frac{1}{\abs{V}} \sum_{v \in V} d\left( \Pi_v^{\sigma_1}\mb{x},\ \Pi_v^{\sigma_2}\mb{x}\right) . \]
	With $i,j \in V$ and $k \in [r]$ as in the definition of switching above, let $B = \{i, \sigma_1^k(i), j, \sigma_1^k(j)\}$. If the $\sigma_1$-graph distance between $v$ and $B$ is $R$, then the pullback names $\Pi_v^{\sigma_1}\mb{x},\ \Pi_v^{\sigma_2}\mb{x}$ agree at least on the radius-$R$ ball centered at $e \in \Gamma$. So
		\[ d\left( \Pi_v^{\sigma_1}\mb{x},\ \Pi_v^{\sigma_2}\mb{x}\right)
			\leq \sum_{l=R+1}^\infty (5r^2)^{-l} 2r(2r-1)^{l-1} 
			\leq \frac{(2/5r)^{R+1}}{1-2/5r} \leq 2 \left( \frac{2}{5r} \right)^R , \]
	and
	\begin{align*}
		\overline{d}(P_{\mb{x}}^{\sigma_1}, P_{\mb{x}}^{\sigma_2})
			&\leq \frac{1}{\abs{V}} \sum_{R=0}^\infty 2 \left( \frac{2}{5r} \right)^R \abs{\{v \in V \st d(v,B) = R\}} \\
			&\leq \frac{1}{\abs{V}} \sum_{R=0}^\infty 2 \left( \frac{2}{5r} \right)^R 4 (2r-1)^{R-1} \\
			&\leq \frac{8}{\abs{V}} \sum_{R=0}^\infty \left( \frac{4}{5} \right)^R \\
			&\leq \frac{2}{\abs{V}}.
	\end{align*}
				
	This implies that if $\sigma_1, \sigma_2$ differ by a switching, the minimum distance (in the transportation metric) from $\mu$ to get any specified number of microstates differs by at most $2/\abs{V}$. Therefore $\rho_\mu(\sigma, h)$ is a Lipschitz function of the permutation $\sigma$.
\end{proof}

We will also need continuity as a function of the measure $\mu$. In fact, $\rho$ is Lipschitz in that sense as well:

\begin{lemma}
\label{lem:rholipschitz}
	For each $h$, $\overline\rho_\mu(h)$ and $\underline\rho_\mu(h)$ are 1-Lipschitz functions of $\mu$.
\end{lemma}
\begin{proof}
	For any $\sigma,h,$ and $\nu$, clearly
		\[ \rho_\mu(\sigma, h) \leq \rho_\nu (\sigma, h) + d(\mu, \nu) , \]
	and by reversing the roles of $\mu,\nu$ we get
		\[ \abs*{ \rho_\mu(\sigma, h) - \rho_\nu (\sigma, h)} \leq d(\mu, \nu) . \]
	The inequality is preserved after taking expectations and a limsup or liminf.
\end{proof}

\subsubsection{Properties of $h^\sharp$, $h^\flat$}

\begin{lemma}
	$h^\sharp(\mu)$ and $h^\flat(\mu)$ are upper semicontinuous functions of $\mu$.
\end{lemma}
\begin{proof}
	Suppose $\mu_k \to \mu$, and let $h = \limsup_{k \to \infty} h^\sharp(\mu_{k})$. Then for any $\delta>0$ we have $h^\sharp(\mu_{k}) \geq h-\delta$ for infinitely many $k$, so $\underline\rho_{\mu_{k}}(h-\delta) = 0$ for infinitely many $k$, and by continuity of $\underline\rho$ (Lemma~\ref{lem:rholipschitz})
		\[ \underline\rho_\mu(h-\delta) = \lim_{k \to \infty} \underline\rho_{\mu_{k}}(h-\delta) = 0 . \]
	This implies $h-\delta \leq h^\sharp(\mu)$, so since $\delta>0$ is arbitrary, $\limsup_{k \to \infty} h^\sharp(\mu_{k}) \leq h^\sharp(\mu)$.
	
	The same argument shows that $h^\flat$ is upper semicontinuous.
\end{proof}

\begin{lemma}
\label{lem:fupperbound}
	For every $\mu$, $h^\sharp(\mu) \leq \overline{\h}^\ann(\mu)$ and $h^\flat(\mu) \leq \underline{\h}^\ann(\mu)$.
\end{lemma}
\begin{proof}
	
	For any $\delta>0$, by Markov's inequality
		\[  \PP\left\{ \frac{1}{\abs{V_n}} \log \abs*{\Omega(\sigma_n, \mu, \varepsilon)} > \overline{\h}^\ann(\mu) + \delta \right\} \leq \EE\abs*{\Omega(\sigma_n, \mu, \varepsilon)} \, e^{-\abs{V_n}(\overline{\h}^\ann(\mu) + \delta)} \]
	which, by definition of $\overline{\h}^\ann$, for all small enough $\varepsilon>0$ is bounded above by $e^{\abs{V_n}(\overline{\h}^\ann(\mu) + \delta/2)} e^{-\abs{V_n}(\overline{\h}^\ann(\mu) + \delta)} = e^{-\abs{V_n} \delta/2}$ for all large enough $n$. In particular, for all large enough $n$
	\begin{align*}
		\EE\rho_\mu(\sigma_n, \overline{\h}^\ann(\mu+\delta))
			&\geq \varepsilon\, \PP\left\{ \rho_\mu(\sigma_n, \overline{\h}^\ann(\mu+\delta)) \geq \varepsilon \right\}  \\
			&= \varepsilon \, \PP\left\{ \frac{1}{\abs{V_n}} \log \abs*{\Omega(\sigma_n, \mu, \varepsilon)} \leq \overline{\h}^\ann(\mu) + \delta \right\} \\
			&\geq \varepsilon \, (1-e^{-\abs{V_n} \delta/2}) .
	\end{align*}
	This implies that $\underline\rho_\mu(\overline{\h}^\ann(\mu)+\delta) > 0$, which, since $\delta>0$ is arbitrary, implies $\overline{\h}^\ann(\mu) \geq h^\sharp(\mu)$.
	
	The proof for lower entropies is the same, but with ``for all large enough $n$'' replaced by ``for infinitely many $n$'' and $\underline{\rho}_\mu$ replaced by $\overline{\rho}_\mu$.
\end{proof}


\subsubsection{Consequences of second-moment condition}

Recall the ``second moment condition'' is that $\overline{\h}^\ann(\lambda) \leq 2 \underline{\h}^\ann(\mu)$ for all self-joinings $\lambda$ of $\mu$.

\begin{lemma}
\label{lem:secondmoment}
	If $\mu$ satisfies the second-moment condition
	then
			\[ \lim_{\varepsilon \to 0} \liminf_{n \to \infty} \frac{1}{\abs{V_n}} \log \frac{\left(\EE \abs{\Omega(\sigma_n, \mu, \varepsilon)}\right)^2}{\EE\left[ \abs*{ \Omega(\sigma_n, \mu, \varepsilon)}^2 \right]} 
			= \lim_{\varepsilon \to 0} \limsup_{n \to \infty} \frac{1}{\abs{V_n}} \log \frac{\left(\EE \abs{\Omega(\sigma_n, \mu, \varepsilon)}\right)^2}{\EE\left[ \abs*{ \Omega(\sigma_n, \mu, \varepsilon)}^2 \right]} 
			= 0 . \]
\end{lemma}
\begin{proof}
	Note that $\frac{\left(\EE \abs{\Omega(\sigma_n, \mu, \varepsilon)}\right)^2}{\EE\left[ \abs*{ \Omega(\sigma_n, \mu, \varepsilon)}^2 \right]} \leq 1$ by Jensen's inequality, so
		\[ \limsup_{\varepsilon \to 0} \liminf_{n \to \infty} \frac{1}{\abs{V_n}} \log \frac{\left(\EE \abs{\Omega(\sigma_n, \mu, \varepsilon)}\right)^2}{\EE\left[ \abs*{ \Omega(\sigma_n, \mu, \varepsilon)}^2 \right]}
		\leq \limsup_{\varepsilon \to 0} \limsup_{n \to \infty} \frac{1}{\abs{V_n}} \log \frac{\left(\EE \abs{\Omega(\sigma_n, \mu, \varepsilon)}\right)^2}{\EE\left[ \abs*{ \Omega(\sigma_n, \mu, \varepsilon)}^2 \right]}
		\leq 0 , \]
	and it suffices to show that
		\[ \liminf_{\varepsilon \to 0} \liminf_{n \to \infty} \frac{1}{\abs{V_n}} \log \frac{\left(\EE \abs{\Omega(\sigma_n, \mu, \varepsilon)}\right)^2}{\EE\left[ \abs*{ \Omega(\sigma_n, \mu, \varepsilon)}^2 \right]} \geq 0. \]
		
	 So assume for the sake of contradiction that for some $\eta > 0$ there are arbitrarily small $\varepsilon>0$ with
		\[ \liminf_{n \to \infty} \frac{1}{\abs{V_n}} \log \frac{\left(\EE \abs{\Omega(\sigma_n, \mu, \varepsilon)}\right)^2}{\EE\left[ \abs*{ \Omega(\sigma_n, \mu, \varepsilon)}^2 \right]} \leq -\eta . \]
	By definition of $\underline{\h}^\ann$, for every $\delta>0$ if $\varepsilon$ is small enough then
		\[ \frac{1}{\abs{V_n}} \log \EE \abs{\Omega(\sigma_n, \mu, \varepsilon)} \geq \underline{\h}^\ann(\mu) - \delta \]
	for all large $n$, so
		\[ \limsup_{n \to \infty} \frac{1}{\abs{V_n}} \log \EE\left[ \abs*{ \Omega(\sigma_n, \mu, \varepsilon)}^2 \right] \geq 2 (\underline{\h}^\ann(\mu) - \delta) + \eta , \]
	and
		\[ \inf_{\varepsilon>0} \limsup_{n \to \infty} \frac{1}{\abs{V_n}} \log \EE\left[ \abs*{ \Omega(\sigma_n, \mu, \varepsilon)}^2 \right] \geq 2 \underline{\h}^\ann(\mu) + \eta . \]
	The left-hand side is equal to the supremum of $\overline{\h}^\ann(\lambda)$ where $\lambda$ ranges over self-joinings of $\mu$, so this implies that the second-moment criterion is violated.
	
	Note also that the inequality does hold if $\eta=0$, and the supremum is attained for some $\lambda$, so there must be a self-joining $\lambda$ with $\overline{\h}^\ann(\lambda) = 2 \underline{\h}^\ann(\mu)$.
\end{proof}

\begin{prop}
\label{prop:secondmomentgeneral}
	If $\mu$ satisfies the second-moment condition then $h^\flat(\mu) = \underline{\h}^\ann(\mu)$ and $h^\sharp(\mu) = \overline{\h}^\ann(\mu)$.
\end{prop}
\begin{proof}
	By the Paley-Zygmund inequality, for any $\delta \in (0,1)$
		\[ \PP \left\{ \abs*{\Omega(\sigma_n, \mu, \varepsilon)} > (1-\delta) \EE \abs*{\Omega(\sigma_n, \mu, \varepsilon)} \right\} \geq \delta^2 \frac{\EE[\abs*{\Omega(\sigma_n, \mu, \varepsilon)}]^2}{\EE[\abs*{\Omega(\sigma_n, \mu, \varepsilon)}^2]} . \]
	By Lemma~\ref{lem:secondmoment}, the second-moment condition implies that for any $\eta>0$, for arbitrarily small $\varepsilon>0$ we have
		\[ \frac{\EE[\abs*{\Omega(\sigma_n, \mu, \varepsilon)}]^2}{\EE[\abs*{\Omega(\sigma_n, \mu, \varepsilon)}^2]} \geq e^{-\eta \abs{V_n}} \]
	for all large enough $n$.
	\begin{enumerate}
		\item (Lower entropies are equal)
		By definition we have
		\[  \EE \abs*{\Omega(\sigma_n, \mu, \varepsilon)} \geq e^{\abs{V_n} (\underline{\h}^\ann(\mu) - \delta)} \]
	for all large $n$. Hence for any $\eta, \delta > 0$, for arbitrarily small $\varepsilon>0$
		\[ \PP\{ \rho_\mu(\sigma_n, \underline{\h}^\ann(\mu)-2\delta) \leq \varepsilon\} \geq \PP\{ \tfrac{1}{\abs{V_n}} \log \abs{\Omega(\sigma_n, \mu, \varepsilon)} > \underline{\h}^\ann(\mu) - \delta\} \geq \delta^2 e^{-\eta \abs{V_n}} \]
	for all large $n$. If for some $t>0$ we had $\overline{\rho}_\mu(\underline{\h}^\ann(\mu)-2\delta) > \varepsilon+t$, then Lemma~\ref{lem:rhoconcentration} would imply
	\begin{align*}
		\PP\{ \rho_\mu(\sigma_n, \underline{\h}^\ann(\mu)-2\delta) \leq \varepsilon\}
			&\leq 2 \exp\left( \frac{-C\abs{V_n} t^2}{4} \right)
	\end{align*}
	for infinitely many $n$, so the above lower bound on the same probability implies $t^2 \leq 4\eta/C$. This means that $\overline{\rho}_\mu(\underline{\h}^\ann(\mu)-2\delta) \leq \varepsilon+2\sqrt{\eta/C}$. Taking $\varepsilon$ and then $\eta$ to 0 gives $\overline{\rho}_\mu(\underline{\h}^\ann(\mu)-2\delta)$, which means $h^\flat(\mu) \geq \underline{\h}^\ann(\mu)-2\delta$. Since $\delta>0$ is arbitrary, this shows $h^\flat(\mu) = \underline{\h}^\ann(\mu)$.
	
		\item (Upper entropies are equal)
		By definition we have
		\[  \EE \abs*{\Omega(\sigma_n, \mu, \varepsilon)} \geq e^{\abs{V_n} (\overline{\h}^\ann(\mu) - \delta)} \]
	for infinitely many $n$. Hence for any $\eta, \delta > 0$, for arbitrarily small $\varepsilon>0$
		\[ \PP\{ \rho_\mu(\sigma_n, \overline{\h}^\ann(\mu)-2\delta) \leq \varepsilon\} \geq \PP\{ \tfrac{1}{\abs{V_n}} \log \abs{\Omega(\sigma_n, \mu, \varepsilon)} > \overline{\h}^\ann(\mu) - \delta\} \geq \delta^2 e^{-\eta \abs{V_n}} \]
	for infinitely many $n$. If for some $t>0$ we had $\underline{\rho}_\mu(\overline{\h}^\ann(\mu)-2\delta) > \varepsilon+t$, then Lemma~\ref{lem:rhoconcentration} would imply
	\begin{align*}
		\PP\{ \rho_\mu(\sigma_n, \overline{\h}^\ann(\mu)-2\delta) \leq \varepsilon\}
			&\leq 2 \exp\left( \frac{-C\abs{V_n} t^2}{4} \right)
	\end{align*}
	for all large $n$, so $t^2 \leq 4\eta/C$. This means that $\underline{\rho}_\mu(\overline{\h}^\ann(\mu)-2\delta) \leq \varepsilon+2\sqrt{\eta/C}$. Taking $\varepsilon$ and then $\eta$ to 0 gives $\underline{\rho}_\mu(\overline{\h}^\ann(\mu)-2\delta)$, which means $h^\sharp(\mu) \geq \overline{\h}^\ann(\mu)-2\delta$. Since $\delta>0$ is arbitrary, this shows $h^\sharp(\mu) = \overline{\h}^\ann(\mu)$. \qedhere
	\end{enumerate}
\end{proof}

In particular, since in the case of $\PP^\unif$ the upper and lower annealed entropies are equal and $\f(\mu \times \mu) = 2\f(\mu)$, we get the following simpler version:

\begin{cor}
\label{cor:2ndmomentf}
	If the product self-joining of $\mu$ has maximal $\f$ invariant, then $\f(\mu) = h^\flat(\mu) = h^\sharp(\mu)$.
\end{cor}

%


\section{Typical local limits for nearest-neighbor interactions}
\label{sec:locallimits}

First, we introduce three modes of convergence. Though the notation is different, these all essentially appear in \cite[Definition 2.3]{montanari2012}. A slight difference is that we consider our permutation graphs to come with directed and labeled edges, which fixes a particular identification of a vertex neighborhood with a neighborhood of the identity in $\Gamma$.

If $\sigma \in \Hom(\Gamma, \Sym(V))$ and $\zeta \in \Prob(\A^V)$, the average empirical distribution of $\zeta$ is defined by
	\[ P_\zeta^\sigma = \frac{1}{\abs{\A^V}} \sum_{\mb{x} \in \A^V} \zeta\{\mb{x}\} \, P_{\mb{x}}^\sigma \in \Prob^\Gamma(\A^\Gamma). \]
If $\{\sigma_n \in \Hom(\Gamma, \Sym(V_n))\}_{n=1}^\infty$ is a sequence of homomorphisms and $\zeta_n \in \Prob(\A^{V_n})$ for each $n$, then we say the sequence $\{\zeta_n\}_{n=1}^\infty$ converges to $\mu \in \Prob^\Gamma(\A^\Gamma)$ locally on average (LOA) if $P_{\zeta_n}^{\sigma_n}$ converges to $\mu$ in the weak topology.

In the same setting, we say the local limit of $\{\zeta_n\}_{n=1}^\infty$ is $\mu$ if for every open neighborhood $\calO$ of $\mu$ in $\Prob(\A^\Gamma)$, 
	\[ \lim_{n \to \infty} \frac{1}{\abs{V_n}} \abs*{\{ v \in V_n \st (\Pi_v^{\sigma_n})_* \zeta_n \in \calO\}} = 1. \]
That is, we require most of the local marginals of $\zeta_n$ to be close to $\mu$, not just their average. The individual local marginals are not necessarily shift-invariant.

More generally, we say the local limit is $\mathfrak{m} \in \Prob(\Prob(\A^\Gamma))$ if
	\[ \lim_{n \to \infty} \frac{1}{\abs{V_n}} \sum_{v \in V} \delta_{(\Pi_v^{\sigma_n})_*\zeta_n} = \mathfrak{m} . \]
If the local limit in this sense is a point mass $\delta_\mu$, we just say the local limit is $\mu$; this is equivalent to the previous definition. This mode of convergence is easier to work with because we can always pass to subsequential limits. In examples below, we will then show that the subsequential limits are point masses. \\

Given an energy function $u \colon \A^\Gamma \to \RR$, to each $\sigma \in \Hom(\Gamma, \Sym(V))$ we associate the finitary Gibbs state $\xi \in \Prob(\A^V)$ defined by
	\[ \xi\{\mb{x}\} \propto \exp \left( - \abs{V} \int u\, dP_{\mb{x}}^\sigma \right) = \exp \left( -\sum_{v \in V} u(\Pi_v^\sigma \mb{x}) \right). \]
This is also called the Boltzmann distribution. We interpret $u(\Pi_v^\sigma \mb{x})$ as the energy of interactions attributed to the vertex $v$, so the sum $\sum_{v \in V} u(\Pi_v^\sigma \mb{x})$ is the total energy. Note that the Gibbs state prefers individual microstates with low energy.

We say $u$ is a nearest-neighbor interaction if it is of the form
	\[ u(\mb{x}) = h(\mb{x}(e)) + \frac{1}{2} \sum_{\gamma \sim e} J(\mb{x}(e),\, \mb{x}(\gamma)) \]
for a function $h \colon \A \to \RR$ and a symmetric function $J \colon \A^2 \to \RR$. We include the factor of $\frac{1}{2}$ to attribute to the spin at the identity $e$ half the energy due to interactions with its neighbors.

Given an energy function $u$, we say the typical LOA limit is $\mu$ if for a typical sofic approximation $\Sigma$ (in the sense of Theorem~\ref{thm:main}) the finitary Gibbs states converge LOA to $\mu$. We use analogous terminology for the other modes of convergence. This makes sense for any random sofic approximation, but we mostly consider the uniform permutation model.

A state $\mu \in \Prob^\Gamma(\A^\Gamma)$ is called $\Sigma$-equilibrium if it is a global maximizer of the (upper) $\Sigma$-pressure defined by $\overline{\press}_\Sigma(\mu) = \overline{\h}_\Sigma(\mu) - u(\mu)$, where $u(\mu) = \int u\, d\mu$ is the average energy per site. The pressure captures the competition between the Boltzmann weight of any individual good model for $\mu$, which is determined by $u(\mu)$, and the number of good models for $\mu$, which is measured by $\overline{\h}_\Sigma(\mu)$.

For simplicity we will only consider random sofic approximations for which, like the uniform permutation model, the upper and lower annealed entropies are equal. This is a natural assumption for computing limits, since we do not want different behavior on different subsequences. In this case we will call maximizers of $\overline{\h}^\ann - u = \underline{\h}^\ann - u$ annealed-equilibrium or, in the uniform case, $\f$-equilibrium.

A Gibbs measure is called extreme if it is an extreme point in the convex set of all Gibbs measures for the interaction. This is equivalent to being tail-trivial \cite[Theorem~7.7]{georgii2011} which is equivalent to non-reconstruction \cite[Proposition~15]{mossel2001}.

\begin{prop}
\label{prop:locallimit}
	Consider any exponentially concentrated random sofic approximation such that $\overline{\h}^\ann \equiv \underline{\h}^\ann$.
	
	Let $\calE \subset \Prob^\Gamma(\A^\Gamma)$ be the set of states which are both annealed-equilibrium and Gibbs for a given nearest-neighbor interaction.
	
	Suppose every element of $\calE$ satisfies the second-moment criterion. Then, over a typical sofic approximation, any subsequential LOA limit of the finitary Gibbs states is a mixture of states in $\calE$.
	
	If also every element of $\calE$ is extreme, then any subsequential local limit $\mathfrak{m} \in \Prob(\Prob(\A^\Gamma))$ is supported on the convex hull of $\calE$.
\end{prop}

It is probably true that every annealed-equilibrium state is automatically Gibbs, so the definition of $\calE$ is probably redundant. If the interaction has no hard constraints, this follows from \cite[Theorem A]{shriver2022a}. On the other hand, \cite{barbieri2022a} considered much more general interactions but did not allow for random sofic approximations.

In the uniform case, this should generally provide a tractable approach to computing local limits because $\f$-equilibrium states must be Markov chains \cite[Lemma~6.9]{shriver2023a}, the set of completely homogeneous Markov chains which are Gibbs states is parametrized by ``boundary laws'', and the $\f$-pressure of a Gibbs Markov chain admits a relatively simple expression in terms of its boundary law. See Section~\ref{sec:pottslimit} below.

In examples below, we will actually show that the subsequential local limits are all the same mixture of states in $\calE$, so that the limits exist.

\begin{proof}[Proof of Proposition \ref{prop:locallimit}]
	Since every element of $\calE$ is annealed-equilibrium and satisfies the second-moment criterion, $\calE$ contains all $\Sigma$-equilibrium measures for a typical sofic approximation $\Sigma$, and for any subsequence of $\Sigma$. This is because if $\mu \in \calE$ and $\nu \in \Prob^\Gamma(\A^\Gamma) \setminus \calE$ then, if $\Sigma'$ is a subsequence of $\Sigma$,
	\begin{align*}
		\overline{\h}_{\Sigma'}(\nu)- u(\nu)
			&\leq \overline{\h}_{\Sigma}(\nu) - u(\nu) \tag{$\Sigma'$ subsequence}\\
			&= \h^\sharp(\nu) - u(\nu) \tag{Thm.~\ref{thm:main}}\\
			&\leq \overline{\h}^\ann(\nu) - u(\nu) \tag{Lemma~\ref{lem:fupperbound}}\\
			&< \underline{\h}^\ann(\mu) - u(\mu) \tag{$\nu \not\in \calE$, $\overline{\h}^\ann \equiv \underline{\h}^\ann$}\\
			&= \h^\flat(\mu) - u(\mu) \tag{Cor.~\ref{cor:2ndmomentf}}\\
			&= \underline{\h}_\Sigma(\mu) - u(\mu) \tag{Thm.~\ref{thm:main}} \\
			&\leq \underline{\h}_{\Sigma'}(\mu) - u(\mu) \tag{$\Sigma'$ subsequence} \\
			&\leq \overline{\h}_{\Sigma'}(\mu) - u(\mu) \tag{$\underline{\h} \leq \overline{\h}$}
	\end{align*}
	The claim about LOA limits then follows from \cite[Theorem~6.5]{shriver2023a}. Note that we are only applying this theorem to individual deterministic sofic approximations, so we do not need to worry about the rather strict definition of local limits over random sofic approximations given in \cite{shriver2023a}.
	
	Now suppose all elements of $\calE$ are extreme. If $\mathfrak{m}$ is a subsequential local limit then it is supported on the set of Gibbs measures and its barycenter $\bar\mu$ is the LOA limit over the same subsequence. Since, as we showed above, $\bar\mu$ is in the convex hull of $\calE$, $\mathfrak{m}$ must be supported on the convex hull of $\calE$.
\end{proof}

The following subsections work out a few examples in the uniform case.

\subsection{Ferromagnetic Ising model}

Let $\A = \{-1, +1\}$. The energy function $u$ for the Ising model can be defined by taking $h \equiv 0$ and $J(\ta_1, \ta_2) = -J\ta_1 \ta_2$ for some fixed interaction strength $J > 0$.

For low interaction strengths ($J \leq J_u = \acoth(2r-1) = \frac{1}{2} \log \frac{r}{r-1}$) there is a unique Gibbs measure. For greater values of $J$ there are exactly three Gibbs Markov chains: one is the free-boundary state $\mu^{FB}$, which we already considered in Section~\ref{sec:KSintro} above. There are also plus- and minus-boundary states $\mu^+, \mu^-$ which are biased towards $+1$ and $-1$ respectively, and each of which can be obtained from the other by interchanging $+1$ and $-1$.

\begin{prop}
\label{prop:isinglimit}
	For the uniform permutation model, the typical local limit of the Ising model is $(\mu^+ + \mu^-)/2$.
\end{prop}

Note that this result is weaker than \cite{montanari2012}, which establishes the local limit for an \emph{arbitrary} sofic approximation. We include it to show how little model-specific work is necessary to establish the typical local limit in this case.

\begin{proof}
	We only need to consider temperatures below the uniqueness threshold -- if there is a unique Gibbs state then it must be the local limit over any sofic approximation.
	
	We first find the elements of $\calE$. Any $\f$-equilibrium state is a Markov chain, and the only Gibbs Markov chains are $\mu^+, \mu^-, \mu^{FB}$. But $\mu^{FB}$ is not $\f$-equilibrium \cite[Proposition~1.1]{shriver2023a}. So $\mu^+, \mu^-$ are the unique $\f$-equilibrium states.
	
	Since $\mu^+$ and $\mu^-$ are extreme \cite[Theorem 12.31(b)]{georgii2011}, any subsequential local limit $\mathfrak{m}$ of the finitary Gibbs states must be supported on convex combinations of $\mu^+, \mu^-$. But every neighborhood marginal of any of the finitary states must be symmetric in $+1/-1$, so $\mathfrak{m}$ must be supported on Gibbs states which have this symmetry. The only such convex combination of $\mu^+$ and $\mu^-$ is $(\mu^+ + \mu^-)/2$. Since this is the only possible subsequential limit, it must be the true limit.
\end{proof}

Note that for the analogous upgrade from LOA to local convergence in \cite{montanari2012} they do not use the geometric property of extremality. Instead they show that the plus and minus states uniquely minimize the energy among Gibbs states, and use this fact in the same way.

\subsection{Antiferromagnetic Ising model}
At every temperature, there is a unique completely homogeneous Markov chain $\mu^\ast$ which is a Gibbs state \cite[p.~253]{georgii2011}. So $\mu^*$ is always the unique $\f$-equilibrium state.

As in the ferromagnetic case, the reconstruction threshold is given by the Kesten--Stigum bound \cite[bottom of page 3]{mezard2006}. Combined with Theorem~\ref{thm:KSbound}, this implies
\begin{prop}
\label{prop:afisinglimit}
	In the non-reconstruction regime, $\mu^\ast$ is the typical local limit of the antiferromagnetic Ising model. If instead $\mu^\ast$ is reconstructible, then it is not weakly typical.
\end{prop}
Since $\mu^\ast$ is ergodic, it would have to be typical to be the typical local limit: by \cite[Corollary 5.7]{austin2016}, if a sequence of measures converges locally to $\mu^*$, then the measures are mostly supported on good models for $\mu^*$. In particular, the sofic approximation must have good models for $\mu^*$.

\subsection{Potts model}
\label{sec:pottslimit}

Let $\A = [q] = \{1, 2, \ldots, q\}$, and fix an interaction strength $J > 0$. This defines a nearest-neighbor interaction with $h \equiv 0$ (no external field) and pair interaction
	\[ J(\ta, \tb) = \left\{ \begin{array}{ll}
					-J, & \ta = \tb \\
					J, & \ta \ne \tb.
					\end{array} \right. \]

Let $Q$ be the $q \times q$ matrix with $e^J$ in the diagonal entries and $e^{-J}$ everywhere else. This is called the ``transfer matrix'' for the interaction.

Given a ``boundary law'' $\ell = (\ell_1, \ldots, \ell_q) \in (0,+\infty)^q$, let $B$ be the $q \times q$ matrix
	\[ B = \left(\begin{array}{ccc} \ell_1 &  & 0 \\  & \ddots &  \\0 &  & \ell_q \end{array}\right) Q \left(\begin{array}{ccc} \ell_1 &  & 0 \\  & \ddots &  \\0 &  & \ell_q \end{array}\right) . \]
After normalizing so that the entries sum to 1, we can think of this as a probability distribution on $[q] \times [q]$. Let $\mu_\ell$ be the $[q]$-valued $\FF_r$-indexed Markov chain with edge marginals proportional to $B$.

By \cite[Corollary 12.17]{georgii2011}, every Markov chain which is a Gibbs measure for the Potts interaction can be obtained by some choice of $\ell$. Multiplying $\ell$ by a positive scalar does not change $\mu_\ell$; if we normalize to fix $\ell_q=1$, then the Markov chain $\mu_\ell$ is Gibbs if and only if $\ell$ is a solution of 
	\[ \ell_i = \left( \frac{[\ell Q]_i}{[\ell Q]_q} \right)^{2r-1}  \quad \forall\ i\in [q] \tag{\cite[Equation 12.16]{georgii2011}} . \]
	
For every $J$, one solution is $\ell \equiv 1$: the corresponding Markov chain $\mu_*$ is called the disordered (or paramagnetic) state. A fairly complete description of the other solutions can be found for example in \cite{kulske2014, galanis2016a}, but here we only introduce the most relevant solutions.

If we fix $i \in [q]$ and look for a solution with $\ell_i = \ell^*$ and $\ell_j = 1$ for $j \ne i$, we get up to two solutions depending on $J$ (aside from $\ell^* = 1$). We call the solution with greater $\ell^*$ the $i$th \emph{ordered state} $\mu_i$ (also called ferromagnetic). All $q$ ordered states are equivalent up to permutations of the symbols $\{1, \ldots, q\}$. In particular, they all have the same entropy and pressure.

Let $J_p = \frac{1}{2} \log \left(\frac{q-2}{(q-1)^{1-\frac{1}{r}}-1}\right)$; this is sometimes called the ordered-disordered threshold.

\begin{prop}
\label{prop:pottslimit}
	 In the uniform permutation model, if $J < J_p$ then the Potts typical local limit is the disordered state $\mu_*$. If $J > J_p$ then the typical local limit is the symmetric mixture of the ordered states.
	 
	 If $J = J_p$ then, over a typical sofic approximation, every subsequential local limit is supported on $\mu_*$ and the symmetric mixture of the ordered states.
\end{prop}

Comparing the values of $\press_{\f}(\mu_\ell)$ for different solutions $\ell$ is difficult to do rigorously.
In this section we will use the results of involved computations in \cite{galanis2016a}. Note that the published version \cite{galanis2016} is abridged and does not contain many of the results we need here.

Consider the function $\Phi(\ell)$ defined by
	\[ \Phi(\ell) = r \log \frac{\ell Q \ell^T}{\norm{\ell}_p^2} \]
where $p = \frac{2r}{2r-1}$. The following lemma allows us to use $\Phi$ to compare values of $\press_{\f}$:
\begin{lemma}
\label{lem:phireduction}
	If $\mu_\ell$ is a Gibbs state then $\press_{\f}(\mu_\ell) = \Phi(\ell)$.
\end{lemma}
See Section \ref{sec:phireductionproof} for a proof.

Maximizing $\Phi$ is still fairly involved, but the work has already been done in \cite{galanis2016a}:
\begin{theorem}[\cite{galanis2016a}]
	For $J < J_p$ the boundary law $\ell \equiv 1$ corresponding to the disordered state is the unique maximizer of $\Phi$. For $J > J_p$ the boundary laws corresponding to any of the ordered states are the only maximizers of $\Phi$.
\end{theorem}

\begin{cor}
	Among Gibbs Markov chains, for $J < J_p$ the disordered state uniquely maximizes $\press_{\f}$. For $J > J_p$ the ordered states are the only maximizers of $\press_{\f}$.
\end{cor}
\begin{proof}
	We only need to compare values of $\press_{\f}(\mu_\ell)$ for $\ell$ such that $\mu_\ell$ is Gibbs, in which case $\press_{\f}(\mu_\ell) = \Phi(\ell)$.
\end{proof}

The last ingredient is that the maximizers of $\press_{\f}$ are tail-trivial.

\begin{prop}
	The Potts ordered states are tail-trivial.
\end{prop}
This is proven in \cite[Lemma 42]{galanis2016a}.

The following result is stated in \cite[Proposition 2.6]{coja-oghlan2023} but no proof seems to have been published. It can be proven using the same method as the previous result.

\begin{prop}
	The disordered state is tail-trivial for all $J < \frac{1}{2} \log \left( 1 + \frac{q}{2r-2} \right)$. In particular, it is tail-trivial for all $J < J_p$.
\end{prop}
As pointed out in \cite{coja-oghlan2023}, $\frac{1}{2} \log \left( 1 + \frac{q}{2r-2} \right)$ is the threshold up to which the messages $\ell \equiv 1$ corresponding to the disordered state are a local maximum of $\Phi$. For the Ising case ($q=2$) this coincides with the uniqueness threshold, so in general this is a sufficient but not necessary condition for tail-triviality.
\begin{proof}
	Since the disordered state corresponds to $\ell \equiv 1$, its edge marginal is simply the transfer matrix $Q$ except normalized so that the entries sum to 1. The Markov transition matrix can be obtained from this by normalizing the columns:
		\[ \frac{1}{e^J + (q-1)e^{-J}} \left(\begin{array}{ccc}e^{J} &  & e^{-J} \\  & \ddots &  \\ e^{-J} &  & e^{J}\end{array}\right) . \]
	
	As in \cite[Lemma 42]{galanis2016a}, it suffices to show that the maximum total variation distance (half the $1$-norm) between two columns is less than $\frac{1}{2r-1}$. For any two distinct columns, the distance is $\frac{e^{2J}-1}{e^{2J}+q-1}$. Solving
		\[ \frac{e^{2J}-1}{e^{2J}+q-1} < \frac{1}{2r-1} \]
	for $J$ yields the condition
		\[ J < \frac{1}{2} \log \left( 1 + \frac{q}{2r-2} \right). \]
	Since $J_p \leq \frac{1}{2} \log \left( 1 + \frac{q}{2r-2} \right)$ for $q \geq 2$, the result follows.
\end{proof}

We can now prove the claimed typical local limits of the Potts model:

\begin{proof}[Proof of Proposition~\ref{prop:pottslimit}]
	The above shows that $\calE$ is $\{\mu_*\}$ if $J<J_p$, and the set of the $q$ ordered states if $J>J_p$.
	
	Let $\mathfrak{m}$ be a subsequential local limit along some typical sofic approximation. Since every element of $\calE$ is tail-trivial (extreme), $\mathfrak{m}$ is supported on the convex hull of $\calE$ (Proposition~\ref{prop:locallimit}). In the disordered case $J <J_p$, this finishes the proof since $\calE$ has only one element.
	
	For the ordered case $J> J_p$, as with the Ising example above, $\mathfrak{m}$ must be supported on states which are invariant under permuting the labels in $[q]$. The symmetric mixture of the ordered states is the only convex combination of them with this property, so it is the only possible subsequential local limit.
	
	For the case $J = J_p$, the set $\calE$ contains the disordered state and all $q$ ordered states. There are two states in the convex hull of $\calE$ which are invariant under permutations of $[q]$: the disordered state and the symmetric combination of the ordered states. So every subsequential local limit $\mathfrak{m}$ is supported on these two states, but we are not able to say how much weight is given to each.
\end{proof}

\subsection{Proof of Lemma~\ref{lem:phireduction}}
\label{sec:phireductionproof}

As in the definition of $\Phi$, let $p = \frac{2r}{2r-1}$.

By \cite[Equation 12.13]{georgii2011}, if $\mu_\ell$ is Gibbs then the single-site marginal of $\mu_\ell$ is $\alpha = \frac{1}{Z_\alpha} (\ell_1^p, \ldots,\ell_q^p)$, and the edge marginal is the measure $\beta \in \Prob(\A \times \A)$ given by $\beta_{ij} = \frac{1}{Z_\beta} q_{ij} \ell_i \ell_j$, where $q_{ij}$ is the $(i,j)$ entry of the transfer matrix $Q$. Note that the assumption that $\mu_{\ell}$ is Gibbs ensures that both marginals of $\beta$ are $\alpha$.

The average energy per site $u(\mu_\ell)$ is given by $r \sum_{i,j = 1}^q \beta_{ij} (-\log q_{ij})$.

Since $\mu_\ell$ is a Markov chain, we have
	\[ \f(\mu_\ell) = (1-2r) \shent(\alpha) + r \shent(\beta) . \]
	
Putting this all together gives
\begin{align*}
	\press_{\f}(\mu_\ell)
		&= (1-2r) \sum_{i=1}^q -\alpha_i \log \alpha_i + r \sum_{i,j=1}^q -\beta_{ij} \log \beta_{ij} + r \sum_{i,j = 1}^q \beta_{ij} \log q_{ij} \\
		&= (1-2r) \sum_{i=1}^q -\alpha_i (p \log \ell_i - \log Z_\alpha) - r \sum_{i,j=1}^q \beta_{ij} (\log \beta_{ij}- \log q_{ij}) .
\end{align*}
By definition of $\beta_{ij}$, and using that both marginals of $\beta$ are $\alpha$, we get
\begin{align*}
	\sum_{i,j=1}^q \beta_{ij} (\log \beta_{ij}- \log q_{ij})
		&= \sum_{i,j=1}^q \beta_{ij} (\log \ell_i + \log \ell_j - \log Z_\beta) \\
		&= \sum_{i=1}^q \log \ell_i \sum_{j=1}^q \beta_{ij} + \sum_{j=1}^q \log \ell_j \sum_{j=1}^q \beta_{ij}- \log Z_\beta \\
		&= 2 \sum_{i=1}^q \alpha_i \log \ell_i- \log Z_\beta.
\end{align*}
So
	\[ \press_{\f}(\mu_\ell) = (1-2r)\log Z_\alpha + r \log Z_\beta - (1-2r) p \sum_{i=1}^q \alpha_i \log \ell_i - 2r \sum_{i=1}^q \alpha_i \log \ell_i .\]
Since $p = \frac{2r}{2r-1}$, the last two terms cancel each other out, leaving
	\[ \press_{\f}(\mu_\ell) = (1-2r)\log Z_\alpha + r \log Z_\beta = r \log \frac{Z_\beta}{Z_\alpha^{p/2}} = \Phi(\ell). \]


\section{Kesten--Stigum bound implies atypicality}

In this section, we only consider the uniform permutation model, since we rely on a result of \cite{bordenave2019} which is proved in that context.

The argument is loosely based on \cite[Theorem 3.1, Corollary 3.2]{lyons2017}, a proof attributed to Allan Sly that the Ising free boundary state is not a factor of iid below the reconstruction threshold. The connection between this approach and \cite{bordenave2019} was suggested by Tim Austin.

A similar correlation bound for factor of iid processes was proven in \cite{backhausz2015} using a very similar kind of spectral argument. An alternative approach to proving Theorem~\ref{thm:KSbound} may be to show that graphings of typical processes are Ramanujan and then apply the main result of \cite{backhausz2015}.

The following theorem was shared with me by Tim Austin, but is likely known to others. See for example \cite{burton2020} for more discussion of connections between sofic entropy and notions of weak containment.

\begin{theorem}
\label{thm:repopbound}
	Suppose $\A$ is a finite alphabet and $\mu \in \Prob^{\FF_r}(\A^{\FF_r})$. If $\mu$ is weakly typical then for any $A \in \CC \FF_r$
		\[ \norm{\pi(A)}_{L^2_0(\mu)} \leq \norm{\lambda(A)}_{\ell^2(\FF_r)} . \]
\end{theorem}

Here $L^2_0(\mu)$ is the space of square-integrable complex-valued functions with mean zero, $\pi$ is the Koopman representation, and $\lambda$ is the left-regular representation. These representations naturally extend to the group ring $\CC \FF_r$.

An equivalent conclusion of the theorem is that the Koopman representation on $L^2_0(\mu)$ is weakly contained in the left-regular representation of $\Gamma$ \cite[Theorem~F.4.4]{bekka2008}.

\begin{proof}
	Suppose $f \in L^2_0(\mu)$ takes only finitely many distinct values, say $\B \Subset \CC$, and let $\nu \in \B^{\FF_r}$ be the factor of $\mu$ generated by $f$. Since $\mu$ is weakly typical, so is $\nu$ \cite[Lem.~4.3 and Prop.~4.10]{austin2016}. Assume that $f$ is not $\mu$-almost-surely zero.

	Let $\sigma_n \colon \FF_r \to \Sym(V_n)$ be a uniformly random homomorphism. Let $\lambda_n$ be the induced (random) unitary representation on $\ell^2(V_n)$. Write $A = \sum a_i \gamma_i \in \CC \FF_r$. 

	Given $\mb{x} \in \B^{V_n}$,
	\begin{align*}
		\norm{\lambda_n(A) \mb{x}}^2
			&= \sum_{v \in V_n} \left( \sum_{i} a_i \, \mb{x}(\sigma_n^{\gamma_i} \cdot v) \right)^2 \\
			&= \abs{V_n} \int G_A \, d P_{\mb{x}}^{\sigma_n}
	\end{align*}
	where $G_A \colon \CC^{\FF_r} \to \CC$ is given by
		\[ G_A(\mb{z}) = \left( \sum_i a_i \mb{z}(\gamma_i) \right)^2 . \]
	Since $G_A$ is continuous, there is a weak-open neighborhood $\calO \ni \nu$ such that if $\mb{x} \in \Omega(\sigma_n, \calO)$ then
		\[ \abs*{\frac{\norm{\lambda_n(A) \mb{x}}^2}{\abs{V_n}} - \int G_A\, d\nu} < \varepsilon .\]
	Since $f$ has mean 0, we can assume that $\mb{x} \cdot \mb{1} = 0$, and $\norm{\mb{x}}^2 \approx \abs{V_n} \Var(\nu_e)$. This means that if there is some $\mb{x} \in \Omega(\sigma_n, \calO)$ then
		\[ \norm{\lambda_n(A)}_0^2 \geq \frac{\norm{\lambda_n(A) \mb{x}}^2}{\norm{\mb{x}}^2} \geq \frac{1}{\Var(\nu_e)} \int G_A\, d\nu - \varepsilon = \frac{\norm{\pi(A) f}^2_{L^2(\mu)}}{\norm{f}_{L^2(\mu)}^2} - \varepsilon \]
	where $\norm{\lambda_n(A)}_0$ is the operator norm of $\lambda_n(A)$ on $\CC^{V_n}$ restricted to the orthogonal complement of the all ones vector $\mb{1}$.
	
	Now since $\nu$ is weakly typical, the probability that $\Omega(\sigma_n, \calO)$ is nonempty has limit superior 1. This proves that for any $\varepsilon>0$
		\[ \limsup_{n \to \infty} \PP\left\{ \norm{\lambda_n(A)}_0 \geq \frac{\norm{\pi(A) f}_{L^2(\mu)}}{\norm{f}_{L^2(\mu)}} - \varepsilon \right\} = 1 \]

	But \cite{bordenave2019} implies that $\norm{\lambda_n(A)}_0$ converges to $\norm{\lambda(A)}$ in probability. (The convergence for $A$ is not explicitly stated in this terminology as a numbered theorem, since $A$ is a polynomial in the generators rather than a self-adjoint linear combination. See the discussion in \cite[Section 1.3]{bordenave2019}.) So we have
		\[ \norm{\lambda(A)} \geq \frac{\norm{\pi(A) f}_{L^2(\mu)}}{\norm{f}_{L^2(\mu)}} . \]
	Since $f$ was chosen arbitrarily from a dense subspace of $L^2_0(\mu)$, this proves the theorem.
\end{proof}

To obtain from this the Kesten--Stigum bound, we now consider the averaging operators $A_m = \sum_{\gamma \in S_m} \frac{1}{\abs{S_m}} \gamma$ where $S_m = \{v \in \FF_r \st d(e,v)=m\}$. We first compute $\norm{\lambda(A_m)}$, then establish a lower bound on $\norm{\pi(A_m)}$.

\begin{prop}
\label{prop:treeavgbound}
	For every $m \in \NN$,
		\[ \norm{\lambda(A_m)}_{\ell^2(\FF_r)} = \left( m+1-\frac{m}{r}\right) \frac{1}{(2r-1)^{m/2}} . \]
\end{prop}

\begin{proof}
	$\abs{S_m} \lambda(A_m)$ is essentially the sum operator of \cite[Section 12.2]{einsiedler2017}, but on a different graph: let $G_m$ be the graph with vertex set $\FF_r$ and an edge between every pair of vertices whose distance in $\FF_r$ is exactly $m$ (deleting the original edges).
	
	Within the proof of the upper bound of \cite[Theorem 12.23]{einsiedler2017}, the following is shown: if $\Lambda$ is a positive-valued function of the directed edges of $G_m$ such that $\Lambda(v,w) = (\Lambda(w,v))^{-1}$ then
		\[ \norm{ \abs{S_m} \lambda(A_m)}_{\ell^2} \leq \rho \coloneqq \sup_{v \in \FF_r} \sum_{w \sim_{G_m} v} \Lambda(v,w) \]
	where the sum is over vertices adjacent to $v$ in $G_m$.
	
	Our desired upper bound will come from taking
		\[ \Lambda(v,w) = (2r-1)^{(d(e,v) - d(e,w))/2} \]
	where $e \in \FF_r$ is the identity and the distance here refers to the original graph structure on $\FF_r$. We will bound $\sum_{w \sim v} \Lambda(v,w)$ by breaking up the sum according to $v$ and $w$'s nearest common ancestor (again, in the original graph): we say their nearest ancestor is $i$ above $v$ if the point on the simple path from $v$ to $w$ which comes closest to $e$ is $i$ units closer to $e$ than $v$ is. Breaking up the sum according to $i$ in this way, we get
	\begin{align*}
		\sum_{w \sim v} \Lambda(v,w)
			&= \sum_{i=0}^m \sum_{w \st \cdots} \Lambda(v,w) \\
			&= \sum_{i=0}^m \abs{\{w \st \cdots\}} \,(2r-1)^{(2i-m)/2}
	\end{align*}
	We bound this by noting that the set of $w$ whose nearest ancestor with $v$ is $i$ above $v$ and whose distance from $v$ is $n$ is a subset of the vertices which are $(m-i)$th generation descendants of that common ancestor. First suppose $v$ is not the identity. If $i=0$, there are $(2r-1)^{m}$ of them; otherwise there are at most $(2r-2)(2r-1)^{m-i-1}$ (or none if $v$ is closer than distance $i$ to the identity). This gives
	\begin{align*}
		\sum_{w \sim v} \Lambda(v,w)
			&\leq (2r-1)^{m} (2r-1)^{(-m)/2} + \sum_{i=1}^{m-1} (2r-2)(2r-1)^{m-i-1} \, (2r-1)^{(2i-m)/2} + (2r-1)^{m/2}\\
			&= (m+1-\tfrac{m-1}{2r-1})(2r-1)^{m/2} .
	\end{align*}
	If $v$ is the identity, then we can compute
		\[ \sum_{w \sim e} \Lambda(e,w) = 2r(2r-1)^{m-1} (2r-1)^{-m/2} = \tfrac{2r}{2r-1} (2r-1)^{m/2} \]
	but for $r \geq 2$ this is smaller than the upper bound in the $v \ne e$ case for all $m \geq 1$. So we get $\rho \leq (m+1-\tfrac{m-1}{2r-1})(2r-1)^{m/2}$,
	which implies the claimed upper bound on $\norm{\lambda(A_m)}$.
	
	We now prove the lower bound, again based on the $m=1$ case in \cite[Theorem~12.23]{einsiedler2017}. To simplify notation let $b=2r-1$ denote the branching factor. For the remainder of the proof we only consider the original, standard graph structure on $\FF_r$. For each $N \geq 2m$ let $f_N \in \ell^2(\FF_r)$ be given by
		\[ f(v) = \left\{ \begin{array}{ll}
			b^{-\frac{1}{2}d(v,e)}, & d(v,e) \leq N \\
			0, & \text{else}. \end{array} \right. \]
	A straightforward calculation shows
		\[ \norm{f_N}_{\ell^2}^2 = 1 + N\frac{b+1}{b} , \]
	and if $d(v,e) = n$ where $m \leq n \leq N-m$ then
		\[ \abs{S_m} \lambda(A_m)f(v) = b^{-n/2} b^{m/2} (m+1 - \tfrac{m-1}{b}) \]
	so
	\begin{align*}
		\norm[\big]{\abs{S_m} \lambda(A_m) f}_{\ell^2}^2
			&> \sum_{n=m}^{N-m}\left( b^{-n/2} b^{m/2} (m+1 - \tfrac{m-1}{b}) \right)^2 \abs{\{v \st d(v,e) = n\}} \\
			&= b^m (m+1 - \tfrac{m-1}{b})^2 (N - 2m+1) \tfrac{b+1}{b},
	\end{align*}
	and
		\[ \abs{S_m} \norm{\lambda(A_m)}_{\ell^2} \geq b^{m/2} (m+1 - \tfrac{m-1}{b}) \sqrt{\frac{N-2m+1}{\frac{b}{b+1}+N}}. \]
	Taking $N \to \infty$ gives the lower bound.
\end{proof}

The following lemma will be used to get a lower bound on $\norm{\pi(A_m)}_{L^2_0(\mu)}$ in terms of the second-largest eigenvalue.
\begin{lemma}
\label{lem:linalg}
	Suppose $M$ is a $d \times d$ column-stochastic matrix, let $\theta$ be its second-largest eigenvalue, and let $p$ be a stationary probability vector. Then $M$ has a left $\theta$-eigenvector $f \in \CC^d$ orthogonal to $p$.
\end{lemma}
\begin{proof}
	Let $p \in \RR^d$ be a stationary probability vector for $M$, or in other words a right $1$-eigenvector whose entries sum to 1. Let $U \subset \CC^d$ be the orthogonal complement of $p$. Since $U$ has codimension 1 and $\mathbf{1} \not\in U$, $\CC^d = U \oplus \vspan\{\mathbf{1}\}$.
	
	Since $x^T M p = x^Tp$, multiplication by $M$ on the right preserves this decomposition of $\CC^d$. This implies that $M$ has a left $\theta$-eigenvector in $U$.
\end{proof}

Note: if $1$ has algebraic multiplicity greater than 1 as an eigenvalue of $M$, then we can take $\theta=1$.

\begin{prop}
\label{prop:repopbound}
	Suppose $\mu$ is a completely homogeneous Markov chain on $\FF_r$ with transition matrix $M$, and let $\theta$ be the second-largest eigenvalue of $M$. Then for every $m \in \NN$,
	\[ \norm{\pi(A_m)}_{L^2_0(\mu)} \geq  \abs{\theta}^m . \]
\end{prop}

\begin{proof}
	Let $p \in \RR^{\A}$ be the single-site marginal of $\mu$. Since $p$ is a right $1$-eigenvector of $M$, by Lemma~\ref{lem:linalg} we can pick some $f \in \CC^{\A}$ which is a left $\theta$-eigenvector of $M$ and orthogonal to $p$.
	
	Define $F \in L^2_0(\mu)$ by $F(\mb{z}) = f(\mb{z}(e))$; note $F$ has mean zero since $f$ is orthogonal to $p$. Multiplying $f$ by a scalar if necessary, we can also assume $\norm{F}_{L^2(\mu)} = \Var_p(f) = 1$. Since $\pi(A_m)$ is self-adjoint,
		\[ \norm{\pi(A_m)}_{L^2_0(\mu)} \geq \abs{\langle F,\ \pi(A_m)F\rangle} = \abs*{\frac{1}{\abs{S_m}} \sum_{\gamma \in S_m} \int f(\mb{z}(e))\, \overline{f(\mb{z}(\gamma))}\, \mu(d\mb{z})} . \]
	Since $\mu$ is completely homogenous, every term in the sum has the same value. If $\gamma \in \FF_r$ is distance $m$ from the identity,
	\begin{align*}
		\int f(\mb{z}(e))\, \overline{f(\mb{z}(\gamma))}\, \mu(d\mb{z})
			&= \sum_{\ta_1, \ta_2 \in \A} f(\ta_1) \overline{f(\ta_2)} \, \PP_{\mb{z} \sim\mu}\{\mb{z}(e) = \ta_1,\ \mb{z}(\gamma) = \ta_2\} \\
			&= \sum_{\ta_1, \ta_2 \in \A} f(\ta_1) \overline{f(\ta_2)}  \, p(\ta_1) M^m_{\ta_2, \ta_1} \\
			&= \sum_{\ta_1 \in \A} f(\ta_1) p(\ta_1) \sum_{\ta_2 \in \A} \overline{f(\ta_2)}M^m_{\ta_2, \ta_1} .
	\end{align*}
	Now since $M$ has real entries and $f$ is a left $\theta$-eigenvector,	
		\[ \sum_{\ta_2 \in \A} \overline{f(\ta_2)}M^m_{\ta_2, \ta_1} = \overline{\sum_{\ta_2 \in \A} f(\ta_2) M^m_{\ta_2, \ta_1} } = \overline{\theta^m f(\ta_1)} . \]
	So since $\Var_p(f) = 1$ and $\EE_p(f) = 0$
		\[ \int f(\mb{z}(e))\, f(\mb{z}(\gamma))\, \mu(d\mb{z}) = \overline{\theta^m} \sum_{\ta_1 \in \A} f(\ta_1) \overline{f(\ta_1)} p(\ta_1) = \overline{\theta^m} . \]
	Altogether this gives
		\[ \norm{\pi(A_m)}_{L^2_0(\mu)} \geq \abs{\overline{\theta^m}} = \abs{\theta}^m \]
	as claimed.
\end{proof}

We can now prove the main result of this section:

\begin{proof}[Proof of Theorem~\ref{thm:KSbound}]
We show that if a Markov chain $\mu$ on $\FF_r$ is weakly typical, the preceding results in this section imply that $\abs{\theta}^2(2r-1) \leq 1$.
		
	If $\mu$ is weakly typical then the above implies that for every $m \in \NN$
		\[ \abs{\theta}^m
			\overset{\text{Prop.}~\ref{prop:repopbound}}{\leq} \norm{\pi(A_m)}_{L^2_0(\mu)}
			\overset{\text{Thm.}~\ref{thm:repopbound}}{\leq} \norm{\lambda(A_m)}_{\ell^2(\FF_r)}
			\overset{\text{Prop.}~\ref{prop:treeavgbound}}{=} \left( m+1-\frac{m}{r}\right) \frac{1}{(2r-1)^{m/2}} , \]
	so
		\[ \abs{\theta} \leq \lim_{m \to \infty} \left( \left( m+1-\frac{m}{r}\right) \frac{1}{(2r-1)^{m/2}} \right)^{1/m} = \frac{1}{\sqrt{2r-1}} \]
	which is the claimed bound on $\theta$.
\end{proof}


\section{Typical upper and lower sofic entropy values: Proof of Theorem~\ref{thm:main}}
\label{sec:typicalproof}

Recall that we have a sequence of finite sets $\{V_n\}_{n=1}^\infty$ and for each $n$ a probability distribution on $\Hom(\Gamma,\Sym(V_n))$, such that the sequence of these probability distributions forms an exponentially concentrated random sofic approximation.

Let $\gamma_n \in \Hom(\Gamma,\Sym(V_n))$ denote a random homomorphism whose law is part of this sequence. 

Recall that there is then a sequence of sets of homomorphisms $\{ S_n^\sofic \subset \Hom(\Gamma, \Sym(V_n)) \}_{n=1}^\infty$ with $\PP(S_n^\sofic) \to 1$ such that if $\sigma_n \in S_n$ for each $n$ then $\Sigma = \{\sigma_n \}_{n=1}^\infty$ is a sofic approximation to $\Gamma$. All of the work in this section will go into using exponential concentration to produce another sequence $\{S'_n\}_{n=1}^\infty$ of sets along which the number of good models is controlled. We will then intersect that sequence with $S_n^{\sofic}$. \\

	
	For $\mu \in \Prob^\Gamma(\A^\Gamma)$, $n \in \NN$ and $\varepsilon,\delta>0$ we define the following four sets of ``bad'' homomorphisms in $\Hom(\Gamma, \Sym(V_n))$, which we write as events in terms of $\gamma_n$: if $h^\sharp(\mu)$ and $h^\flat(\mu)$ are finite, then
	\begin{align*}
		\mathcal{F}^{\geq,\sharp}(\mu,\varepsilon,\delta,n)
			&= \left\{ \frac{1}{\abs{V_n}} \log \abs*{\Omega(\gamma_n, \mu, \varepsilon)} \geq h^\sharp + \delta \right\} \\
		\mathcal{F}^{\geq,\flat}(\mu,\varepsilon,\delta,n)
			&= \left\{ \frac{1}{\abs{V_n}} \log \abs*{\Omega(\gamma_n, \mu, \varepsilon)} \geq h^\flat + \delta \right\}  \\
		\mathcal{F}^{\leq,\sharp}(\mu,\varepsilon,\delta,n)
			&= \left\{ \frac{1}{\abs{V_n}} \log \abs*{\Omega(\gamma_n, \mu, \varepsilon)} \leq h^\sharp - \delta \right\} \\
		\mathcal{F}^{\leq,\flat}(\mu,\varepsilon,\delta,n)
			&= \left\{ \frac{1}{\abs{V_n}} \log \abs*{\Omega(\gamma_n, \mu, \varepsilon)} \leq h^\flat - \delta \right\} .
	\intertext{If $h^\flat(\mu) = -\infty$ then we set $\mathcal{F}^{\leq,\flat}(\mu,\varepsilon,\delta,n) = \varnothing$ and in the definition of $\mathcal{F}^{\geq,\flat}(\mu,\varepsilon,\delta,n)$ we change $h^\flat(\mu) + \delta$ to $0$. Make similar changes if $h^\sharp(\mu) = -\infty$.}
	\end{align*}
	We first show that all of these sets have small probability:
	\begin{enumerate}
		\item For each $\delta > 0$ and small enough (depending on $\delta$) $\varepsilon > 0$ there exists $\xi$ such that for all large $n$
			\[ \PP \left( \mathcal{F}^{\geq,\sharp}(\mu,\varepsilon,\delta,n) \right) < e^{-\xi \abs{V_n}} . \]
		\emph{Proof:}
		\begin{enumerate}
			\item in the case $h^\sharp \geq 0$: by definition of $h^\sharp$ we have $\underline\rho_\mu(h^\sharp + \delta) > 0$, and
			\[ \mathcal{F}^{\geq,\sharp}(\mu,\varepsilon,\delta,n) = \left\{ \frac{1}{\abs{V_n}} \log \abs*{\Omega(\gamma_n, \mu, \varepsilon)} \geq h^\sharp + \delta \right\} \subseteq \left\{ \rho_\mu(\gamma_n, h^\sharp + \delta) \leq \varepsilon \right\} \]
		so if $\varepsilon$ is small compared to $\underline\rho_\mu(h^\sharp + \delta)$ then this is exponentially unlikely for all large $n$ by Lemma~\ref{lem:rhoconcentration}. 
			\item in the case $h^\sharp = -\infty$: by definition of $h^\sharp$ we have $\underline\rho_\mu(0) > 0$, and
			\[ \mathcal{F}^{\geq,\sharp}(\mu,\varepsilon,\delta,n) = \left\{ \frac{1}{\abs{V_n}} \log \abs*{\Omega(\gamma_n, \mu, \varepsilon)} \geq 0 \right\} \subseteq \left\{ \rho_\mu(\gamma_n, 0) \leq \varepsilon \right\} \]
		so if $\varepsilon$ is small compared to $\underline\rho_\mu(0)$ then this is exponentially unlikely for all large $n$. 
		\end{enumerate}
		
		\item For each $\delta > 0$ and small enough (depending on $\delta$) $\varepsilon > 0$ there exists $\xi$ such that for infinitely many $n$
			\[ \PP \left( \mathcal{F}^{\geq,\flat}(\mu,\varepsilon,\delta,n) \right) < e^{-\xi \abs{V_n}} . \]
		\emph{Proof:}
		\begin{enumerate}
			\item if $h^\flat \geq 0$: by definition of $h^\flat$ we have $\overline\rho_\mu(h^\flat + \delta) > 0$, and
			\[ \mathcal{F}^{\geq,\flat}(\mu,\varepsilon,\delta,n) = \left\{ \frac{1}{\abs{V_n}} \log \abs*{\Omega(\gamma_n, \mu, \varepsilon)} \geq h^\flat + \delta \right\} \subseteq \left\{ \rho_\mu(\gamma_n, h^\flat + \delta) \leq \varepsilon \right\} \]
		so if $\varepsilon$ is small compared to $\overline\rho_\mu(h^\flat + \delta)$ then this is exponentially unlikely for infinitely many $n$. 
			\item if $h^\flat = -\infty$: by definition of $h^\flat$ we have $\overline\rho_\mu(0) > 0$, and
			\[ \mathcal{F}^{\geq,\flat}(\mu,\varepsilon,\delta,n) = \left\{ \frac{1}{\abs{V_n}} \log \abs*{\Omega(\gamma_n, \mu, \varepsilon)} \geq 0 \right\} \subseteq \left\{ \rho_\mu(\gamma_n, 0) \leq \varepsilon \right\} \]
		so if $\varepsilon$ is small compared to $\overline\rho_\mu(0)$ then this is exponentially unlikely for infinitely many $n$.
		\end{enumerate}
		
		\item For each $\varepsilon, \delta > 0$ there exists $\xi$ such that for infinitely many $n$
			\[ \PP \left( \mathcal{F}^{\leq,\sharp}(\mu,\varepsilon,\delta,n) \right) \leq e^{-\xi \abs{V_n}} . \]
		\emph{Proof:}
		\begin{enumerate}
			\item if $h^\sharp \geq 0$: by definition of $h^\sharp$ we have $\underline\rho_\mu(h^\sharp - \delta) = 0$, and
			\[ \mathcal{F}^{\leq,\sharp}(\mu,\varepsilon,\delta,n) = \left\{ \frac{1}{\abs{V_n}} \log \abs*{\Omega(\gamma_n, \mu, \varepsilon)} \leq h^\sharp - \delta \right\} \subseteq \left\{ \rho_\mu(\gamma_n, h^\sharp - \tfrac{\delta}{2}) \geq \varepsilon \right\} \]
		which is exponentially unlikely for infinitely many $n$.
			\item if $h^\sharp = -\infty$: $\mathcal{F}^{\leq,\sharp}(\mu,\varepsilon,\delta,n)$ is defined to be empty, so the bound holds for all $n$ for any $\xi$.
		\end{enumerate}
		
		\item For each $\varepsilon, \delta > 0$ there exists $\xi$ such that for all large $n$
			\[ \PP \left( \mathcal{F}^{\leq,\flat}(\mu,\varepsilon,\delta,n) \right) \leq e^{-\xi \abs{V_n}} . \]
		\emph{Proof:}
		\begin{enumerate}
			\item if $h^\flat \geq 0$: by definition of $h^\flat$ we have $\overline\rho_\mu(h^\flat - \delta) = 0$, and
			\[ \mathcal{F}^{\leq,\flat}(\mu,\varepsilon,\delta,n) = \left\{ \frac{1}{\abs{V_n}} \log \abs*{\Omega(\gamma_n, \mu, \varepsilon)} \leq h^\flat - \delta \right\} \subseteq \left\{ \rho_\mu(\gamma_n, h^\flat - \tfrac{\delta}{2}) \geq \varepsilon \right\} \]
		which is exponentially unlikely for all large $n$. 
			\item if $h^\flat = -\infty$: $\mathcal{F}^{\leq,\flat}(\mu,\varepsilon,\delta,n)$ is defined to be empty, so the bound holds for all $n$ for any $\xi$.
		\end{enumerate}
	\end{enumerate}
	
	Pick a sequence $(\delta_m)_{m \in \NN}$ decreasing with limit 0. For each $m$, for every $\mu \in \Prob^\Gamma(\A^\Gamma)$ for small enough $\varepsilon>0$
	\begin{enumerate}[\indent(a)]
		\item the bounds in items 1 and 2 above hold
		\item for every $\nu \in \ball{\mu}{\varepsilon}$, $h^\sharp(\nu)$ and $h^\flat(\nu)$ are less than $h^\sharp(\mu) + \delta_m$ and $h^\flat(\mu) + \delta_m$, respectively (or less than 0 if $h^\sharp(\mu)$ or $h^\flat(\mu)$ is $-\infty$). This uses upper semicontinuity of $h^\flat$ and $h^\sharp$.
	\end{enumerate}
	Then, by compactness, for each $m$ we can pick a a finite set of measures $\mu_{m,1}, \ldots, \mu_{m,K_m}$ with corresponding $\varepsilon_{m,1}, \ldots, \varepsilon_{m,K_m}$ such that $\{ \ball{\mu_{m,k}}{\varepsilon_{m,k}}\}_{k=1}^{K_m}$ covers $\Prob^\Gamma(\A^\Gamma)$. We can also require that $\limsup_{m \to \infty} \max\{\varepsilon_{m,k} \st 1 \leq k \leq K_m\} = 0$.
	
	For $\diamond \in \{\leq, \geq\}$  and $\star \in \{ \flat, \sharp\}$ let $\calE^{\diamond,\star}_{m,k,n} = \mathcal{F}^{\diamond,\star}(\mu_{m,k}, \varepsilon_{m,k}, \delta_m, n)$, except for $(\leq,\sharp)$ and $(\geq,\flat)$: in those cases we take this definition for those infinitely many $n$ where the exponentially decaying bound on their probability holds and define the sets to be empty otherwise.
	
	For each $m$ let $\xi_m$ be such that $\PP(\calE^{\diamond,\star}_{m,k,n}) < e^{-\xi_m \abs{V_n}}$ for all large enough $n$ and all choices of $\diamond \in \{\leq,\geq\}$, $\star \in \{\flat, \sharp\}$ and all $k = 1, \ldots, K_m$.
	
	
	Choose an increasing sequence $N_m$ so that for all $n > N_m$
	\begin{align*}
		\frac{1}{m}
			&\geq 4 \sum_{i=1}^m K_i e^{-\xi_i \abs{V_n}}
		\intertext{and}
		\PP \left[ \calE^{\diamond,\star}_{i,k,n} \right]
			&\leq e^{-\xi_i \abs{V_n}} \quad \text{ for all } i \in [m],\ \diamond \in \{\leq,\geq\},\ \star \in \{\flat, \sharp \},\ k \in [K_i] . 
	\end{align*}
	Now let
		\[ \calU_{m,n} = \left( \bigcup_{i=1}^m \bigcup_{k=1}^{K_i} \calE^{\geq,\flat}_{i,k,n} \cup \calE^{\geq,\sharp}_{i,k,n} \cup\calE^{\leq,\flat}_{i,k,n} \cup\calE^{\leq,\sharp}_{i,k,n} \right)^c . \]
	For all $n > N_m$ we have
		\[ \PP \left( \calU_{m,n} \right) \geq  1 - 4 \sum_{i=1}^m K_i e^{-\xi_i \abs{V_n}}\geq 1 - \frac{1}{m} . \]
		
	Now for each $m$ and each $n \in (N_m, N_{m+1}]$ let $S_n' = \calU_{m,n}$. Note that for each $m,n$ we have $\calU_{m,n} \supset \calU_{m+1,n}$. So if $\Sigma = \{\sigma_n\}_{n = 1}^\infty$ is such that $\sigma_n \in S_n'$ for each $n$, this ensures that for each $m$, for all $n > N_m$ we have $\sigma_n \in \calU_{m,n}$. In particular, for any $\mu$:
	\begin{enumerate}
		\item Given $m \in \NN$ pick $k \in [K_m]$ such that $\mu \in \ball{\mu_{m,k}}{\varepsilon_{m,k}}$. For all $n > N_m$ we have $\sigma_n \not\in \calE^{\geq,\sharp}_{m,k,n}$, so for all large $n$ we have $\sigma_n \not\in \calF^{\geq,\sharp}(\mu_{m,k}, \varepsilon_{m,k}, \delta_m, n)$. Hence if $\varepsilon$ is small enough that $\ball{\mu}{\varepsilon} \subset \ball{\mu_{m,k}}{\varepsilon_{m,k}}$ then
			\[ \limsup_{n \to \infty} \frac{1}{\abs{V_n}} \log \abs*{\Omega(\sigma_n, \mu, \varepsilon)} \leq  \limsup_{n \to \infty} \frac{1}{\abs{V_n}} \log \abs*{\Omega(\sigma_n, \mu_{m,k}, \varepsilon_{m,k})} \leq h^\sharp(\mu_{m,k}) + \delta_m , \]
		where we take $-\infty + \delta_m = -\infty$.
		Taking $\varepsilon\to 0$ then gives
			\[ \overline\h_\Sigma(\mu) \leq h^\sharp(\mu_{m,k}) + \delta_m, \]
		and taking the lim sup as $m \to \infty$ gives $\overline\h_\Sigma(\mu) \leq h^\sharp(\mu)$ by upper semicontinuity of $h^\sharp$.
		
		If $h^\sharp(\mu) = -\infty$, then this gives $\overline\h_\Sigma(\mu) = h^\sharp(\mu)$.
		
		\item For each $m$, pick $k$ as in part 1. For all $n > N_m$ we have $\sigma_n \not\in \calE^{\geq,\flat}_{m,k,n}$. So for infinitely many $n$ we have $\sigma_n \not\in \calF^{\geq,\flat}(\mu_{m,k}, \varepsilon_{m,k}, \delta_m, n)$, and if $\varepsilon$ is small enough
			\[ \liminf_{n \to \infty} \frac{1}{\abs{V_n}} \log \abs*{\Omega(\sigma_n, \mu, \varepsilon)} \leq h^\flat(\mu_{m,k}) + \delta_m . \]
		Taking $\varepsilon \to 0$ then $m \to \infty$ gives $\underline\h_\Sigma(\mu) \leq h^\flat(\mu)$, by upper semicontinuity of $h^\flat$.
		
		If $h^\flat(\mu) = -\infty$, then this gives $\underline\h_\Sigma(\mu) = h^\flat(\mu)$.
		
		\item Given $\varepsilon>0$, pick $m$ large enough that if we pick $k$ as in part 1 then $\ball{\mu_{m,k}}{\varepsilon_{m,k}} \subset \ball{\mu}{\varepsilon}$; this is possible because we required all $\varepsilon_{m,k}$ to be uniformly small for large $m$, so whichever $\ball{\mu_{m,k}}{\varepsilon_{m,k}}$ contains $\mu$ must eventually be contained in $\ball{\mu}{\varepsilon}$. This ensures that $h^\sharp(\mu) \leq h^\sharp(\mu_{m,k}) + \delta_m$.
		
		For all $n > N_m$ we have $\sigma_n \not\in \calE^{\leq,\sharp}_{m,k,n}$. So for infinitely many $n$ we have $\sigma_n \not\in \calF^{\leq,\sharp}(\mu_{m,k}, \varepsilon_{m,k}, \delta_m, n)$, and
			\begin{align*}
				\limsup_{n \to \infty} \frac{1}{\abs{V_n}} \log \abs*{\Omega(\sigma_n, \mu, \varepsilon)}
					&\geq \limsup_{n \to \infty} \frac{1}{\abs{V_n}} \log \abs*{\Omega(\sigma_n, \mu_{m,k}, \varepsilon_{m,k})} \\
					&\geq h^\sharp(\mu_{m,k}) - \delta_m \\
					&\geq h^\sharp(\mu) - 2\delta_m.
			\end{align*}
		Taking $m \to \infty$ then $\varepsilon \to 0$ gives $\overline\h_\Sigma(\mu) \geq h^\sharp(\mu)$.
		
		\item Given $\varepsilon>0$, pick $m$ as in part 3. For all $n > N_m$ we have $\sigma_n \not\in \calE^{\leq,\flat}_{m,k,n}$. So for all large $n$ we have $\sigma_n \not\in \calF^{\leq,\flat}(\mu_{m,k}, \varepsilon_{m,k}, \delta_m, n)$, and
		\begin{align*}
			\liminf_{n \to \infty} \frac{1}{\abs{V_n}} \log \abs*{\Omega(\sigma_n, \mu, \varepsilon_m)}
				&\geq \liminf_{n \to \infty} \frac{1}{\abs{V_n}} \log \abs*{\Omega(\sigma_n, \mu_{m,k}, \varepsilon_{m,k})} \\
				&\geq h^\flat(\mu_{m,k}) - \delta_m \\
				&\geq h^\flat(\mu) - 2\delta_m.
		\end{align*}
		Taking $m \to \infty$ then $\varepsilon \to 0$ gives $\underline\h_\Sigma(\mu) \geq h^\flat(\mu)$.
	\end{enumerate}

	Finally, if we take $S_n = S_n' \cap S_n^\sofic$ for each $n$, then $\{S_n\}_{n=1}^\infty$ is a sequence with the desired property.
\printbibliography

\end{document}